\documentclass[12pt]{amsart}
\usepackage{amssymb}
\usepackage{amsmath}
\usepackage[english]{babel}

\newtheorem{thm}{Theorem}

\newtheorem{lemma}{Lemma}

\newtheorem{prop}{Proposition}
\theoremstyle{definition}

\newtheorem{remk}{Remark}
\newtheorem{exam}{Example}
\newcommand{\set}[1]{\left\{#1\right\}}

\newcommand{\cX}{{\mathfrak{X}}}
\newcommand{\cA}{\mathcal{A}}
\newcommand{\cB}{{\mathcal{B}}}

\newcommand{\cD}{{\mathcal{D}}}
\newcommand{\cE}{{\mathcal{E}}}
\newcommand{\cJ}{{\mathcal{J}}}
\newcommand{\cC}{{\mathcal{C}}}
\newcommand{\cF}{{\mathcal{F}}}
\newcommand{\cH}{{\mathcal{H}}}

\newcommand{\cS}{{\mathcal{S}}}

\newcommand{\cY}{{\mathcal{Y}}}
\newcommand{\cW}{{\mathcal{W}}}
\newcommand{\cg}{{\mathfrak{g}}}

\newcommand{\dC}{{\mathbb{C}}}
\newcommand{\dN}{{\mathbb{N}}}
\newcommand{\dR}{{\mathbb{R}}}
\newcommand{\dZ}{{\mathbb{Z}}}
\newcommand{\gn}{\mathfrak{n}}
\newcommand{\gm}{\mathfrak{m}}
\newcommand{\gk}{\mathfrak{k}}
\newcommand{\gr}{\mathfrak{r}}

\newcommand{\gI}{\mathfrak{I}}
\newcommand{\gX}{\mathfrak{X}}

\newcommand{\Lin}{{\mathrm{Lin}}}

\title{Algebras of Fractions and Strict Positivstellens\"atze for $*$-Algebras}
\begin{document}
\author{Konrad Schm\"udgen}
\address{Universit\"at Leipzig, Mathematisches Institut, Johannisgasse 26, 04103 Leipzig, Germany}
\email{schmuedgen@math.uni-leipzig.de}



\keywords{Positivstellensatz, fractions, Weyl algebra, enveloping algebra, integrable  representations}

\maketitle
\begin{abstract}
In this paper we investigate a $*$-algebra $\cX$ of fractions
associated with a unital complex $*$-algebra $\cA$. The algebra $\cX$ and its Hilbert space representations are used to prove abstract noncommutative strict Positivstellens\"atze for $\cA$. Multi-grading of $\cA$ are studied as technical tools to verify the assumptions of this theorem.

As applications we obtain new strict Positivstellens\"atze for the Weyl algebra  and for the Lie algebra $\cg$ of the affine group of the real line. We characterize integrable representations of the Lie algebra $\cg$ in terms of resolvents of the generators and derive a new integrability criterion for representations of $\cg$.

\end{abstract}
\section{Introduction}
\evensidemargin = -5pt%
Positivstellens\"atze are fundamental results of real algebraic geometry \cite{PD}, \cite{M1}. They represent positive or nonnegative polynomials on semi-algebraic sets in terms of weighted sums of squares of polynomials.
Noncommutative strict Positivstellens\"atze have been proved for the Weyl algebra in \cite{S3} (see also \cite{C}) and for the enveloping algebra of a finite dimensional Lie algebra in \cite{S4}. The technical ingredients of these proofs are Hilbert space representations of certain algebras of fractions. Results of this kind can be considered as steps towards a noncommutative real algebraic geometry (see e.g. \cite{S5} and \cite{HP} for recent surveys).

In the present paper we investigate a fraction $*$-algebra $\cX$ associated with a unital $*$-algebra $\cA$. Our main aim is to develop a general method and technical tools for proving noncommutative strict Positivstellens\"atze of $\cA$ by means of the $*$-algebra $\cX$.

Throughout $\cA$ is a complex unital
$*$-algebra which has no zero-divisors and $\cS_O$ is a $*$-invariant left Ore set of $\cA$.
Further, $\cS$ is a unital $*$-invariant countable submonoid of $\cS_O$ and  $\gX $ is a unital $*$-subalgebra
of the fraction $*$-algebra $\cA\cS_O^{-1}$ such that $\cA\subseteq \cX\cS$, $\cX\subseteq \cA\cS^{-1}$ and $\cS^{-1}$ is a right Ore subset of $\cX$. Let $\cS_G$ denote a $*$-invariant set
of generators of $\cS$ and $\cX_s$ the quotient $*$-algebra of $\cX$ by the two-sided $*$-ideal generated by  $s\in\cS$.

Let us explain the contents of the paper. In Section 2 we show how a bounded $*$-representation $\rho$ of $\gX$ satisfying $\ker\rho(s^{-1})=\set{0}$
for  $s \in \cS$ gives rise to an (unbounded) $*$-representation $\pi_\rho$ of the $*$-algebra $\cA$. Despite of being
 essential for the results in Section 3, this construction seems to be  useful
in unbounded representation theory of $*$-algebras.
Representations of the form $\pi_\rho$ are
candidates for the definition of "well-behaved" unbounded representations of the $*$-algebra $\cA$
(see also Remark \ref{wellbeh} below). Section 3
contains three variants of an abstract strict Positivstellensatz for the $*$-algebra $\cA$. Our main abstract strict Positivstellensatz (Theorem \ref{abstpos2}) can be stated as follows. Assume that the $*$-algebra $\cX$ is algebraically bounded and the inner automorphisms $\alpha_s(\cdot)=s{\cdot}s^{-1}$, $s\in \cS$, leave $\cX$ invariant. Let $c$ be a hermitian element of $\cA$ and $t\in \cS$ such that $t^{-1}c(t^*)^{-1}$ is in $ \cX$. If the operators $\pi_\rho(c)$ and $\rho_s(t^{-1}c(t^*)^{-1})$ are strictly positive for all irreducible $*$-representations $\pi_\rho$ of $\cA$ and $\rho_s$ of $\gX_s$ for $s \in \cS_G$, then there exists an element $s \in \cS_O$ such that $scs^*$ is a finite sum of  hermitian squares in the $*$-algebra $\cA$. The fraction algebras and the denominator sets used in \cite{S3} and \cite{S4}  satisfy the assumptions of Theorem \ref{abstpos2}. In general it might be not easy to prove
that these assumptions are fulfilled.
In Section 4 we study multi-graded $*$-algebras and develop some conditions and results which are useful tools to verify the assumptions of Theorem \ref{abstpos2}.

The second group of results of this paper are two
strict Positivstellens\"atze proved in Sections 5 and 7.
The first one (Theorem \ref{pqpos}) is about the Weyl algebra $\cW(1)$ with denominator set $\cS_O{=}\cS$ generated by $\cS_G=\{p\pm \alpha i,q\pm\beta i\}$, where $\alpha$ and $\beta$ are fixed nonzero reals. The proof uses a result of Kato \cite{K2} about the integrability of the canonical commutation relation. The second application (Theorem \ref{abpos}) concerns the enveloping algebra $\cE(\cg)$ of the Lie algebra $\cg$ of the $ax+b$-group. Here the denominator set $\cS_O=\cS$ is generated by $\cS_G=\{i{\sf a}\pm(\alpha {+}n) i, i{\sf b}\pm\beta i;~ n\in\dZ \}$, where $\alpha$ and $\beta$ are reals such that $\alpha< -1$, $\beta \neq 0$ and $\alpha$ is not an integer and $\{ {\sf a},{\sf b}\}$ is a basis of $\cg$ satisfying the Lie relation $[\sf a,\sf b]=\sf b$. The results of Section 6 are essentially used in the proof of the Positivstellensatz in Section 7, but they are also of interest in itself. Section 6 contains a description of integrable representations of the Lie algeba $\cg$ in terms of a fraction algebra (Proposition \ref{dUAB} and Theorem \ref{axbresol}) and a new integrability criterion
(Theorem \ref{intdens}) which is the counterpart of Kato's theorem for representations of the Lie algebra $\cg$.

We close this introduction by  collecting some terminology on $*$-algebras and unbounded representations (see \cite{S1} for  a detailed treatment of this matter).
Suppose that $\cB$ is a unital $*$-algebra.
A \textit{$*$-representation} $\pi$ of $\cB$ on  a dense linear subspace $\cD(\pi)$ of a Hilbert space $\cH(\pi)$
is an algebra homomorphism of $\cB$ into
the algebra of linear operators mapping $\cD(\pi)$ into itself
such that $\pi (1)\varphi=\varphi$ and $\langle
\pi(b)\varphi,\psi\rangle =\langle \varphi,\pi(b^\ast)\psi\rangle$
for $\varphi,\psi\in\cD(\pi)$ and $b\in \cB$. Here $\langle\cdot,\cdot\rangle$ denotes the scalar product of $\cH(\pi)$. The \textit{graph topology} $t_\pi$ is the  locally convex topology on $\cD(\pi)$ defined by the seminorms $\varphi \to ||\pi(b)\varphi ||$, where $b\in \cB$.
Let $\cB_h=\{b \in \cB:b^*=b\}$ be the hermitian part of $\cB$ and let $\sum\cB^2$ be the cone of all finite sums of hermitian squares $b^*b$, where $b \in \cB$. We denote by $\cB_b$ the set of all $b\in \cB$ for which there exists a positive number $\lambda$ such that $\lambda {\cdot} 1 -b^*b \in \sum \cB^2$. Then $\cB_b$ is a $*$-algebra \cite{V}, see e.g. \cite{S3}. We say that $\cB$ is \textit{algebraically bounded} when $\cB=\cB_b$.
We write $T>0$ for a symmetric operator $T$ on a Hilbert space when $\langle T\psi,\psi \rangle >0$ for all nonzero vectors $\psi$ in its domain $\cD(T)$.

\section{Some Algebraic Preliminaries}

First let us fix the algebraic setup used throughout this paper.
We assume that $\cS_O$ is a $*$-invariant left Ore set of $\cA$. This means that $\cS_O$ is a unital $*$-invariant submonoid of $\cA\backslash\set{0}$ (that is, $1\in\cS_O$, $s^*\in\cS_O$ and $st\in\cS_O$ for $s,t\in\cS_O$) satisfying the left Ore condition (that is, for each $s\in\cS_O$ and $a\in\cA$ there exist $t\in\cS_O$ and $b\in\cA$ such that $ta=bs$). The symbol $1$ always denotes the unit element of $\cA$.
The $*$-invariance and the left Ore condition  imply that $\cS_O$ satisfies the
right Ore condition (that is, for any $s\in\cS_O$ and
$a\in\cA$ there are $t\in\cS_O$ and $b\in\cA$ such that
$at=sb$). Let $\cA\cS_O^{-1}$ be the fraction $*$-algebra
with denominator set $\cS_O$ (see e.g. \cite{R}, \cite{GW}).
We denote by $\cS$ a unital $*$-invariant submonoid of $\cS_O$ generated by a countable subset $\cS_g$, by  $\cS_G$ the set $\cS_g\cup \cS_g^*$ and by $\cA_G$ a $*$-invariant set of generators of the algebra $\cA$.

{\textit{Throughout we suppose that $\gX $ is a $*$-subalgebra
of  $\cA\cS_O^{-1}$ such that $\cS^{-1} \subseteq \cX$ and $\cA_G \subseteq \cX\cS$.}} \\
Let $\gX_G$ be  a fixed $*$-invariant sets of algebra generators of $\gX$.
For $s \in\cS$ let $\gI_s$ be the two-sided $*$-ideal of $\gX$ generated by  $s^{-1}$ (that is, $\gI_s= \gX s^{-1}\gX+\gX(s^*)^{-1}\gX$) and by $\gX_s=\gX/\gI_s$ the
corresponding quotient $*$-algebra. For notational simplicity we denote elements of $\cX$ and their images in $\gX_s$ under the canonical map
by the same symbol.

The main assumption used in this paper is the following condition:

$(O)$ \textit{$\cS^{-1}$ is a right Ore set of the algebra $\cX$, that is, for $s \in \cS$ and $x\in\gX$ there exist elements $t \in \cS$ and $y\in\gX$ such
that $xt^{-1}=s^{-1}y$ (or equivalently $sx =yt$).}

The next lemma is often used in what follows. It reformulates the well-known fact (\cite{GW}, Lemma 4.21(a))
 that finitely many fractions can be brought to a common denominator.
\begin{lemma}\label{commonden}
Assume that $(O)$ is satisfied.  Let $\cF$ be a finite subset of $\cS$. There exists an element $t_0 \in \cS$ such that $st^{-1}\in \cX$ and $t^{-1}s \in \gX$ for all $s \in \cF$, where $t=t_0^*t_0$.
\end{lemma}
\begin{proof}
  We first prove by induction on the cardinality that for each finite set $\cF \subseteq \cS$ there exists $t_1 \in \cS$ such that $st_1^{-1} \in \gX$ for all $s\in \cF$. Suppose this is true for $\cF$. Let $s_1 \in \cS$. Since $s_1^{-1}\in \cX$, by assumption $(O)$ there are elements  $t_2 \in \cS$ and $y\in \cX$ such that $s_1^{-1}t_2^{-1}=t_1^{-1}y$. Then we have $s(t_2s_1)^{-1}=(st_1^{-1})y \in \gX$ for $s \in \cF$ and $s_1(t_2s_1)^{-1}=t_2^{-1}\in \gX$ which proves our claim for  $\cF\cup \{s_1\}$.

Now let $\cF$ be a finite subset of $\cS$. Applying the statement proved in the preceding
paragraph to the set $\cF\cup \cF^*$, there exists $t_0\in \cS$ such that $st_0^{-1} \in \gX$ and $s^*t_0^{-1} \in \gX$. Then we have $s(t_0^*t_0)^{-1}=(st_0^{-1})(t_0^{~*})^{-1} \in \gX$ and $(t_0^*t_0)^{-1}s =((s^*t_0^{-1})(t_0^{~*})^{-1})^* \in \gX$ for $s \in \cF$.
\end{proof}
Let $\cX\cS=\{xs ; x\in \cX,s\in S\}$ and $\cS\cX=\{sx  ; x\in \cX, s\in \cS\}$ considered as subsets of $\cA\cS_O^{-1}$. The next lemma collects some equivalent formulations of condition $(O)$. We omit the details of the simple proofs. In the proof of the implication $(iv)\to (v)$ we use Lemma \ref{commonden} in order to show that $\cX\cS$ is closed under addition.
\begin{lemma}\label{reconO}
The following are equivalent:\\
$(i)$ Condition (O) is satisfied.\\
$(ii)$ $\cX\cS=\cS\cX$.\\
$(iii)$ $\cX\cS$ is $*$-invariant.\\
$(iv)$ $\cX\cS$ is closed under multiplication.\\
$(v)$ $\cX\cS$ is a $*$-subalgebra of $\cA\cS_O^{-1}$.
\end{lemma}

Suppose that $(O)$ holds. Because $\cS^{-1}$ is $*$-invariant and a right Ore set of $\cX$ by $(O)$, it is also a left Ore set and the $*$-algebra $\cX(\cS^{-1})^{-1}$ of quotients with denominator set $\cS^{-1}$ exists. Since $\cX$ is a $*$-subalgebra of $\cA\cS_O^{-1}$, it follows from the universal property of algebras of quotients that $\cX(\cS^{-1})^{-1}$ is $*$-isomorphic to the $*$-subalgebra $\cX\cS$ (by Lemma \ref{reconO}) of $\cA\cS_O^{-1}$. As assumed above the $*$-algebra $\cX\cS$ contains the generator set $\cA_G$ of the algebra $\cA$. Therefore,  we have
\begin{align}\label{axs}
\cA \subseteq \cX\cS.
\end{align}

The following three conditions are on sets of generators of $\cS$ and $\cX$. Because of Lemma \ref{A12} below they are convenient tools for the verification of condition $(O)$.

$(IA)$ \textit{For $s \in \cS_G$ and $x\in\gX_G$ there is an element $y\in\gX$ such
that $xs^{-1}=s^{-1}y$.}

$(A1)$ \textit{For $s \in \cS_G$ and $x\in\gX_G$ there exist elements $t \in \cS_G$ and $y\in\gX$ such
that $xt^{-1}=s^{-1}y$.}

$(A2)$ \textit{Given $s_1,s_2 \in \cS_G$, there exists  an  element $t \in \cS$ such that $s_1t^{-1}\in\gX$ and $s_2t^{-1}\in\gX $.}

Note that $(IA)$ is a strengthening of $(A1)$.  An equivalent formulation of $(IA)$ is that for each generator $s \in \cS_G$ (and hence for all $x \in \cS$) the inner automorphism $\alpha_s(x):=sxs^{-1}$ of the algebra $\cA\cS_O^{-1}$ leaves $\cX$ invariant.

\begin{lemma}\label{A12}
$(i)$ If $(A1)$ and $(A2)$ are satisfied, then $(O)$ holds.\\
$(ii)$ If $(IA)$ is fulfilled, then $(A1)$, $(A2)$ and hence $(O)$ are valid.
\end{lemma}
\begin{proof}
(i): Let $\cY$ denote the set of elements $x\in \cX$ such that for each $s \in \cS_G$ there exist $t \in \cS_G$ (!) and $y\in \cX$ satisfying $sx=yt$.
Let $x_1, x_2 \in \cY$ and $s \in \cS_G$. Then there are $t_1, t_2 \in \cS_G$ and $y_1, y_2 \in \cX$ such $sx_1=y_1t_1$ and $sx_2=y_2t_2$.
Since $t_1\in \cS_G$ and $x_2\in \cY$, there exist $t_3\in \cS_G$ and $y_3 \in \cX$ such that $t_1x_2=y_3t_3$. Then we have $sx_1x_2 =y_1t_1x_2 =y_1y_3t_3$, so that $x_1x_2 \in \cY$.
Because $\cY$ contains the set $\cX_G$
of algebra generators by $(A1)$, it follows that ${\Lin}~ \cY =\cX$.

Since
$t_1,t_2 \in \cS_G$, condition (A2) applies and there exists $t \in \cS$ such that $t_1t^{-1}, t_2 t^{-1} \in \cX$. Then we have $s(\lambda_1x_1+\lambda_2x_2)= (\lambda_1 y_1t_1t^{-1}+\lambda_2y_2t_2t^{-1})t\in \cX{\cdot}t$ for $\lambda_1,\lambda_2\in \dC$. This proves that  $(O)$ is valid for generators $s \in \cS_G$ and for all $x \in{\Lin}~\cY=\cX$.

Now suppose  $s_1$ and $s_2$ are elements of $\cS$ such that the assertion of $(O)$ holds for all elements of $\cX$. Therefore, if $x\in \cS$, then there are $t_1, t_2\in \cS$ and $y_1,y_2\in \cX$ such that $s_1x=y_1t_1$ and $s_2y_1=y_2t_2$. Then, $s_2s_1x= s_2y_1t_1=y_2t_2t_1$, that is, $(O)$ holds for the product $s_1s_2$ and all $x \in \cX$ as well. Hence condition $(O)$ is valid for arbitrary elements $s \in \cS$ and $x\in \cX$.

(ii): Trivially, $(IA)$ implies $(A1)$. Let $s_1, s_2 \in \cS_G$. Putting $t=s_1s_2$, we have $s_1t^{-1}=
\alpha_{s_1}(s_2^{-1}) \in \cX$ as follows from $(A3)$ and $s_2t^{-1}=s_1^{-1}\in \cX$. This proves $(A2)$.
\end{proof}

{\it{Throughout the rest of this paper we assume that assumption $(O)$ is satisfied.}}

Now let $\rho$ be a $*$-representation of $\cX$. Since $\rho$ is a right $\cX$-module and $\cS^{-1}$ is a right Ore set,
$$
\cD_{tor}(\rho):= \{\varphi \in \cD(\rho): {\rm There~ exists}~ s \in \cS ~{\rm such~ that}~ \rho(s^{-1})\varphi=0 \}
$$
is a linear subspace of $\cD(\rho)$ which is invariant under $\rho$ (\cite{GW}, Lemma 4.12). Hence the restriction $\rho_{tor}$ of $\rho$ to the $t_\rho$-closure of $\cD_{tor}(\rho)$ in $\cD(\rho)$ is a $*$-representation $\rho_{tor}$ of $\cX$ called the {\it{$\cS^{-1}$-torsion subrepresentation}}
of $\rho$. We say that $\rho$ is {\it{$\cS^{-1}$-torsion}} if $\cD_{tor}(\rho)=\cD(\rho)$ and that $\rho$ is {\it{$\cS^{-1}$-torsionfree}} if $\cD_{tor}(\rho)=\{0\}$. We shall omit the prefix $\cS^{-1}$ if no confusion can arise.

Suppose now that $\rho$ is a {\it{bounded}} $*$-representation of $\cX$ on a Hilbert space $\cH(\rho)=\cD(\rho)$. Then $\cD(\rho_{tor})$ is closed subspace of $\cH(\rho)$ and $\rho$ is a direct sum of the {\it{torsion subrepresentation}} $\rho_{tor}$ on the Hilbert space $\cD(\rho_{tor})$ and a {\it{torsionfree subrepresentation}} $\rho_{tfr}$ on $\cD(\rho_{tfr}):=\cH(\rho)\ominus \cD(\rho_{tor})$.
\begin{lemma}\label{decomp}
Suppose that condition (IA) is satisfied. Then each bounded  $*$-representation $\rho$ of $\gX$ on a Hilbert space $\cH(\rho)=\cD(\rho)$ decomposes into a direct sum $\rho = \rho_{tfr}\oplus ( \oplus_{s\in \cS_G} \rho_s)$ of $*$-representations $\rho_{tfr}$ and  $\rho_s$ of $\gX$ such that $\rho_{tfr}$ is torsionfree and $\rho_s(s^{-1})=0$ for $s \in \cS_G$, so $\rho_s$ factors to a $*$-representation of the $*$-algebra $\gX_s=\gX/ \gI_s$.
\end{lemma}
\begin{proof}
We enumerate the countable
set $\cS_G$ as $\cS_G=\{r_j;j \in N\}$, where either $N=\{1,\cdots,m\}$ with $m \in \dN$ or $N=\dN$. Put $\cH_{r_1}:= \ker~\rho(r_1^{-1})$. Let $x \in \cX$. Since $\cS_G$ is $*$-invariant, $r_1^*\in \cS_G$ and hence $y:=r_1^*x^*(r_1^*)^{-1}\in \cX$ by  $(IA)$, so that
$\rho(r_1^{-1})\rho(x)\varphi=\rho(y^*)\rho(r_1^{-1})\varphi=0$ for
$\varphi\in\cH_{r_1}$. That is, the (bounded) $*$-representation $\rho$ leaves $\cH_{r_1}$ invariant.
Let $\rho_{r_1}$ and $\tilde{\rho}$ denote the restrictions of $\rho$ to $\cH_{r_1}$ and $\cH(\rho)\ominus \cH_{r_1}$, respectively. Then we have
$\rho_{r_1}(r_1^{-1})=0$ and $\ker~\tilde{\rho}(r_1^{-1})=\{0\}$ by definition.
 Proceeding in similar manner by induction, we obtain an orthogonal direct sum of
$*$-representations $\rho_{r_j}$ of $\gX$ on subspaces $\cH_{r_j}$, $j \in N$. Clearly, for the restriction $\rho_{tfr}$ of $\rho$ to the invariant subspace $\cH_0:=\cH(\rho)\ominus(\oplus_{j\in N}\cH_{r_j})$ we have ${\rm ker}\rho(s^{-1})=\set{0}$  for $s \in \cS_G$ and hence for all $s\in \cS$. This means that $\rho_{tfr}$ is torsionfree.
\end{proof}

Let us illustrate the preceding decomposition by a very simple example.
\begin{exam}\label{exam1}
Let $\cA=\dC[x]$ be the $*$-algebra of polynomials in a hermitian variable $x$. Set
$\cS_G= \{ s{:=}x^2{+}1
 \}$ and $\cS=\cS_O= \{ s^n; n \in \dN_0 \}$. Let $\cX$ be the unital $*$-subalgebra of $\cA\cS_O^{-1}$ generated by $a{:=}s^{-1}$ and $b{:=}xs^{-1}$. It is not difficult to show that each $*$-representation $\rho$ of $\cX$ is of the form
\begin{align*}
\rho(p(a,b))=\int_\cC p(\lambda, \mu) dE(\lambda,\mu),~~ p\in \dC[a,b],
\end{align*}
for some spectral measure $E$ on $\cH(\rho)$ supported on the circle $\cC$ given by the equation $\lambda^2+\mu^2=\lambda$. Then we have $\cD(\rho_{tor})=\cD(\rho_s)=E((0,0))\cH(\rho)$,
$\cD(\rho_{tfr})=E(\cC \backslash (0,0))\cH(\rho)$ and $\rho_s(p(a,b))\varphi=p(0,0)\varphi$ for $\varphi \in E((0,0))\cH(\rho).$ Note that $b^2\in \cJ_s$ and $b \notin \cJ_s$, but $\rho_s(b)=0$ (see e.g. Lemmas \ref{t2} and \ref{rho} below).
\end{exam}

\section{Representations of $\cA$ Associated with Torsionfree Representations of $\gX$}

Suppose that  $\rho$ is a {\it bounded  $\cS^{-1}$-torsionfree} $*$-representation of the $*$-algebra $\gX$ on a Hilbert
space $\cD(\rho)=\cH$. That $\rho$ is torsionfree means that $\ker\rho(s^{-1})=\set{0}$ for all $ s \in \cS$.
Our  aim is to associate an (unbounded) $*$-representation $\pi_\rho$ of the $*$-algebra $\cA$ with $\rho$.
Define
\begin{align}\label{domrho}
\cD_\rho=\cap_{s\in\cS}~\rho(s^{-1})\cH.
\end{align}

\begin{lemma}\label{drho}
$(i)$ $\cD_\rho$ is dense in the Hilbert space $\cH.$\\
$(ii)$ $\rho(x)\cD_\rho\subseteq\cD_\rho$ for  $x\in\gX.$\\
$(iii)$ $\rho(s^{-1})\cD_\rho=\cD_\rho$ for $s\in\cS$.
\end{lemma}
\begin{proof}
$(i)$: The main technical tool for proving this assertion is the so-called {\it Mittag-Leffler lemma} (see e.g. \cite{S1}, p. 15). Let us develop the necessary setup for this result.

We enumerate the countable set $\cS_G$ of generators as $\cS_G=\{r_j;j \in N\}$ such that $r_1=1$, where $N=\{1,\cdots,m\}$ with $m \in \dN$ or $N=\dN$. For $n\in \dN$ let $\cS^n$ denote the set of all products
$r_{j_1}\dots r_{j_n},$ where $j_1\leq n,\dots,j_r\leq n$ and $j_1,\dots,j_n \in N$. Since the set $\cS^n$ is finite, it follows from Lemma \ref{commonden} that for each $n\in \dN$ there exists an element $t_n=t_n^* \in \cS$
such that $st_n^{-1}\in\gX$ for all $s\in\cS^n$ and $t_{n}t_{n+1}^{-1}\in\gX$. Setting $\cS^0=\{1\}$ and $t_0=1$, the latter is also satisfied for $n=0$.

For $n \in \dN_0$, let $E_n$ denote the vector space $\rho(t_n^{-1})\cH$ equipped with the scalar product defined by $(\varphi,\psi)_n=\langle\rho(t_n^{-1})^{-1}\varphi,\rho(t_n^{-1})^{-1}\psi \rangle$, where $\varphi, \psi \in E_n$. Since $E_n$ is the range of the bounded injective operator $\rho(t_n^{-1})$, $(E_n,( \cdot,\cdot )_n)$ is a Hilbert space with norm $||\varphi||_n =||\rho(t_n^{-1})^{-1}\varphi ||$.

We first show that $E_{n+1}$ is a subspace of $E_n$ and that $||\cdot||_n\leq c_n ||\cdot||_{n+1}$ for some positive constant $c_n$. For let $\psi \in \cH$ and set $\varphi :=\rho(t_{n+1}^{-1})\psi$. Since $t_nt_{n+1}^{-1}\in \cX$ by the choice of elements $t_k$ , we obtain $\varphi =\rho(t_{n+1}^{-1})\psi =\rho(t_n^{-1})\rho(t_nt_{n+1}^{-1})\psi
$
which proves that $E_{n+1} \subseteq E_n$. By definition we have $||\varphi||_{n+1} =||\psi ||$ and hence $
||\varphi||_n =||\rho(t_nt_{n+1}^{-1})\psi||\leq ||\rho(t_nt_{n+1}^{-1})||~ ||\varphi||_{n+1}.$

Next we prove that $E_{n+1}$ is dense in the normed space $(E_n,||\cdot||_n)$.
For this it suffices to show that each vector $\zeta\in E_n $ which is orthogonal to  $E_{n+1}$ in the Hilbert space $(E_n,( \cdot,\cdot )_n)$ is the null vector. Put $\xi:=\rho(t_n^{-1})\zeta$. That the vector $\zeta$ is orthogonal to $E_{n+1}$ means that
\begin{align*}
0&=(\zeta,\rho(t_{n+1}^{-1})\varphi)_n=\langle\rho(t_n^{-1})^{-1}\zeta,\rho(t_n^{-1})^{-1}\rho(t_{n+1}^{-1})\varphi \rangle =\langle \xi, \rho(t_n^{-1})^{-1}\rho(t_n^{-1})\rho(t_nt_{n+1}^{-1})\varphi \rangle \\ & =\langle \xi, \rho(t_nt_{n+1}^{-1})\varphi \rangle = \langle \rho(t_nt_{n+1}^{-1})^* \xi,\varphi \rangle
= \langle \rho(t_{n+1}^{-1}t_n)\xi, \varphi \rangle =\langle \rho(t_{n+1}^{-1}t_n)\rho(t_n^{-1})\zeta,\varphi \rangle =\langle \rho(t_{n+1}^{-1})\zeta,\varphi\rangle
\end{align*}
for all $\varphi \in \cH$, where we freely used the properties of the $*$-representation $\rho$ of $\gX$ and of the larger $*$-algebra $\cA\cS_O^{-1}$. Thus we obtain $\rho(t_{n+1}^{-1}) \zeta=0$.
Since $\rho$ is torsionfree,  $\ker\rho(t_{n+1}^{-1})=\set{0}$. Hence we get $\zeta =0$. This proves that $E_{n+1}$ is dense in $E_n$.

In the preceding two paragraphs we have shown that the assumptions of the Mittag-Leffler lemma (see \cite{S1}, Lemma 1.1.2) are fulfilled. From this result it follows that the vector space $E_\infty := \cap_{n\in \dN_0}E_n$ is dense in the normed space $E_0=\cH$.  Obviously,
$\cD_\rho\subseteq E_\infty.$ Let $s \in \cS$. Then $s \in \cS^n$ for some $n \in \dN$ and hence $\rho(t_n^{-1})\cH =\rho(s^{-1})\rho(st_n^{-1})\cH \subseteq \rho(s^{-1})\cH$. This in turn yields $E_\infty \subseteq \cD_\rho$. Therefore, $\cD_\rho=E_\infty$ is dense in $\cH$.

(ii): Suppose that $\varphi\in\cD_\rho$ and $x\in\gX$. Let
$s\in\cS.$ By assumption $(O)$ there exist elements $t\in \cS$ and $y \in \cX$ such that $sx=yt$, so that $xt^{-1}=s^{-1}y$. From the definition (\ref{domrho}) of $\cD_\rho$, there is a vector $\psi\in\cH$ such that $\varphi =\rho(t^{-1})\psi$. Then we have
$\rho(x)\varphi= \rho(xt^{-1})\psi=\rho(s^{-1})\rho(y)\psi\in \rho(s^{-1})\cH$.
Since $s \in \cS$ was arbitrary, we have shown that $\rho(x)\varphi\in \cap_{s \in \cS}~ \rho(s^{-1})\cH =\cD_\rho.$

(iii): Suppose $s \in \cS$. Since $\rho(s^{-1})\cD_\rho \subseteq \cD_\rho$ by (ii), it suffices to show that each vector $\varphi \in \cD_\rho$ belongs to $\rho(s^{-1})\cD_\rho$. According to the definition of $\cD_\rho$, we have $\varphi \in \rho(s^{-1})\cH$ and $\varphi \in  \rho ((ts)^{-1})\cH$ for each $t \in \cS$, that is, there are vectors $\psi \in \cH$ and $\eta_t \in \cH$ such that $\varphi =\rho(s^{-1})\psi =\rho((ts)^{-1})\eta_t$.
Since $\ker\rho(s^{-1})=\set{0}$, the latter implies that $\psi =\rho(t^{-1})\eta_t$, so  that $\psi \in \cap_{t \in \cS}~ \rho(t^{-1})\cH =\cD_\rho$ and $\varphi =\rho(s^{-1})\psi$.
\end{proof}

Let $a\in\cA.$ Suppose that $s$ is an element of $\cS$ such that $as^{-1}\in\gX.$
From (\ref{axs}) it follows that such an element $s$ always exists. Define
\begin{gather}\label{defpi}
\pi_\rho(a)\varphi:=\rho(as^{-1})\rho(s^{-1})^{-1}\varphi,~~
\varphi\in\cD_\rho.
\end{gather}

\begin{thm}\label{assrep}
Let $\rho$ be a bounded  $\cS^{-1}$-torsionfree $*$-representation of the $*$-algebra $\gX$ on a Hilbert space $\cH$. Then $\pi_\rho$ is a well-defined closed $*$-representation of the $*$-algebra
$\cA$ with Frechet graph topology on the dense domain $\cD(\pi_\rho):=\cD_\rho$ of the Hilbert space $\cH$. For $s \in \cS$ and $\varphi \in \cD(\pi_\rho)$ we have $\pi_\rho(s)\cD(\pi_\rho)=\cD(\pi_\rho)$ and
$\pi_\rho(s)\varphi=\rho(s^{-1})^{-1}\varphi$. The $*$-representation $\pi_\rho$ of $\cA$ is irreducible if and only if the $*$-representation $\rho$ of  $\gX$ is irreducible.
\end{thm}
\begin{proof}
We first show that the operator $\pi_\rho(a)$ is well-defined, that is, $\pi_\rho(a)$  (\ref{defpi}) does not depend on the
particular element $s$ of $\cS$ satisfying $as^{-1}\in\gX.$ Let
$\widetilde{s}\in \cS$ be another element such that $a\widetilde{s}^{-1}\in\gX.$
By Lemma \ref{commonden}  there exists $t\in\cS$ such that
$st^{-1}\in\gX$ and $\widetilde{s}t^{-1}\in\gX.$ Then
$at^{-1}=(as^{-1})(st^{-1})\in\gX.$ Let $r$ denote $s$ or
$\widetilde{s}.$ Writing $\varphi=\rho(t^{-1})\psi$ with
$\psi\in\cH,$ we compute
\begin{gather}\nonumber
\rho(ar^{-1})\rho(r^{-1})^{-1}\varphi=\rho(ar^{-1})\rho(r^{-1})^{-1}\rho(t^{-1})\psi=
\rho(ar^{-1})\rho(r^{-1})^{-1}\rho(r^{-1}rt^{-1})\psi=\\
=\rho(ar^{-1})\rho(rt^{-1})\psi=\rho(ar^{-1}rt^{-1})\rho(t^{-1})^{-1}\psi=\rho(at^{-1})\rho(t^{-1})^{-1}\varphi,
\end{gather}
so $\rho(as^{-1})\rho(s^{-1})^{-1}\varphi=\rho(a\widetilde{s}^{-1})\rho(\widetilde{s}^{-1})^{-1}\varphi$. This shows that the operator $\pi\rho(a)$ is well-defined.

Since $\rho(s^{-1})^{-1}\varphi\in\cD_\rho$ and
$\rho(as^{-1})\rho(s^{-1})^{-1}\varphi\in\cD_\rho$ by Lemma \ref{drho},(ii) and (iii), we have  $\pi_\rho(a)\varphi\in\cD_\rho,$ that is, $\pi_\rho(a)$
maps the domain $\cD(\pi_\rho)$ into itself.

Suppose that $a,b\in\cA.$ We shall prove that $\pi_\rho(a+b)=\pi_\rho(a)+\pi_\rho(b)$ and $\pi_\rho(ab)=\pi_\rho(a) \pi_\rho(b)$.

By (\ref{axs}) there are elements $s_1,s_2 \in \cS$ such that $as_1^{-1}\in \gX$ and $bs_2^{-1}\in \gX$. By Lemma \ref{commonden} we can find $s \in \cS$ such that $s_1s^{-1}\in \gX$ and $s_2s^{-1} \in \gX$. Since then $as^{-1} \in \gX$, $bs^{-1}\in \gX$ and $(a+b)s^{-1}\in \gX$, the relation
$$
\rho(as^{-1})\rho(s^{-1})^{-1}\varphi+\rho(bs^{-1})\rho(s^{-1})^{-1}\varphi=\rho((a+b)s^{-1})\rho(s^{-1})^{-1}\varphi,~\varphi \in \cD_\rho,
$$
says that $\pi_\rho(a+b)=\pi_\rho(a)+\pi_\rho(b)$.

From (\ref{axs}), there exist elements $t_1,t_2,t_3,t_4\in\cS$ such that
$at_1^{-1},bt_2^{-1},abt_3^{-1}\in\gX$ and $t_1bt_4^{-1}\in\gX.$
By Lemma \ref{commonden} there is an element $t\in\cS$ such that
$t_jt^{-1}\in\gX$ for $j=1,2,3,4.$ Then we have
$abt^{-1}=(abt_3^{-1})(t_3t^{-1})\in\gX,\
t_1bt^{-1}=(t_1bt_4^{-1})(t_4t^{-1})\in\gX$ and
$bt^{-1}=(bt_2^{-1})(t_2t^{-1})\in\gX.$ Let $\varphi\in\cD_\rho.$
Inserting the corresponding definitions of $\pi_\rho(ab),\pi_\rho(a)$ and $\pi_\rho(b)$ we derive
\begin{gather*}
\pi_\rho(ab)\varphi=\rho(abt^{-1})\rho(t^{-1})^{-1}\varphi=
\rho(at_1^{-1})\rho(t_1bt^{-1})\rho(t^{-1})^{-1}\varphi\\
=\rho(at_1^{-1})\rho(t_1^{-1})^{-1}\rho(t_1^{-1})\rho(t_1bt^{-1})\rho(t^{-1})^{-1}\varphi=
\pi_\rho(a)\rho(t_1^{-1}t_1bt^{-1})\rho(t^{-1})^{-1}\varphi\\
=\pi_\rho(a)\rho(bt^{-1})\rho(t^{-1})^{-1}\varphi=\pi_\rho(a)\pi_\rho(b)\varphi.
\end{gather*}

Finally, we verify that
$\langle\pi_\rho(a)\varphi,\psi\rangle=\langle\varphi,\pi_\rho(a^+)\psi\rangle$
for $a\in\cA$ and $\varphi,\psi\in\cD_\rho.$ From (\ref{axs}) and Lemma \ref{commonden} there
are elements $t_1,t_2  \in\cS$ and $t{=}t^* \in \cS$  such that $at_1^{-1},a^*t_2^{-1}\in\gX$ and
$t_1t^{-1}, t_2t^{-1}\in\gX.$ Since then $at^{-1}\in\gX$ and
$a^*t^{-1}\in\gX$, using that $\rho(t^{-1})$ is
bounded self-adjoint operator we compute
\begin{gather*}
\langle\pi_\rho(a)\varphi,\psi\rangle=\langle\rho(at^{-1})\rho(t^{-1})^{-1}\varphi,\psi\rangle=
\langle\rho(t^{-1})^{-1}\varphi,\rho((at^{-1})^*)\psi\rangle\\
=\langle\rho(t^{-1})^{-1}\varphi,\rho(t^{-1}a^*)\rho(t^{-1})\rho(t^{-1})^{-1}\psi\rangle=
\langle\rho(t^{-1})^{-1}\varphi,\rho(t^{-1}a^*t^{-1})\rho(t^{-1})^{-1}\psi\rangle\\
=\langle\rho(t^{-1})^{-1}\varphi,\rho(t^{-1})\rho(a^*t^{-1})\rho(t^{-1})^{-1}\psi\rangle=
\langle\rho(t^{-1})\rho(t^{-1})^{-1}\varphi,\pi_\rho(a^*)\psi\rangle=\langle\varphi,\pi_\rho(a^*)\psi\rangle,
\end{gather*}
where we used the fact that $\rho(t^{-1})\cD_\rho =\cD_\rho$ according to Lemma \ref{drho}(iii).

Clearly,  $\pi_\rho(\mathbf{1})\varphi{=}\varphi$ for
$\varphi\in\cD(\pi).$ Recall from Lemma \ref{drho}(i) that $\cD_\rho$ is dense in $\cH$. Thus, we have shown that $\pi_\rho$ is a
$*$-representation of $\cA$ on the dense domain $\cD_\rho$ of the Hilbert space $\cH$.

Let $s\in \cS$. Because $ss^{-1}{=}\mathbf{1} \in \gX$, we have $\pi_\rho(s)\varphi =\rho(s^{-1})^{-1}\varphi$ for $\varphi \in \cD_\rho$. Since $\rho(s^{-1})\cD_\rho=\cD_\rho$ by Lemma \ref{drho}(iii), it follows that $\pi_\rho(s)\cD(\pi_\rho)=\cD(\pi_\rho)$.

 To prove the assertion concerning the graph topology of $\pi_\rho$, we retain the notation from
 the proof of Lemma \ref{drho}(i). Since $\pi_\rho(t_n)\varphi =\rho(t_n^{-1})^{-1}\varphi$ for $\varphi \in \cD_\rho$  and $n \in \dN$, the graph seminorm $||\pi_\rho(t_n)\varphi||$ is just the norm $||\varphi ||_n$. Let $a \in \cA$. Applying once more (\ref{axs}) there is an element $t \in \cS$ such that $at^{-1} \in \gX$. We can find a number $n\in \dN$ such that $t\in \cS^n$. Since then
 $tt_n^{-1} \in \gX$, we have $at_n^{-1}=(at^{-1})(tt_n^{-1}) \in \gX$ and hence
 $$
 ||\pi_\rho(a)\varphi|| = ||\rho(at_n^{-1})\rho(t_n^{-1})^{-1} \varphi||=  ||\rho(at_n^{-1})\pi_\rho(t_n)\varphi||\leq ||\rho(at_n^{-1})||~ ||\varphi||_n.
$$
The preceding shows that the graph topology of $\pi_\rho$ is generated by the family of norms
$||\cdot ||_n$, $n \in \dN$. Hence it is the projective limit topology of the countable family of  Hilbert spaces $(E_n,||\cdot ||_n)$ on $E_\infty =\cap_n E_n =\cD_\rho$. Therefore, the graph topology of $\pi_\rho$ is metrizable and complete. The latter implies in particular that the representation $\pi_\rho$ is closed.

It remains to prove the assertion about the irreducibility. Recall that a $*$-representation $\pi_\rho$ is irreducible if and only $0$ and $I$ are the only projections in the strong commutant $\pi_\rho(\cA)^\prime_s$ (\cite{S1}, 8.3.5). Hence it suffices to show that $\pi_\rho(\cA)^\prime_s$ is equal to the commutant $\rho(\gX)^\prime$. Suppose that $T \in\pi_\rho(\cA)^\prime_s$. By definition $T$ maps $\cD(\pi_\rho)$ into itself and we have $T\pi_\rho(a)\varphi=\pi_\rho(a)T\varphi$ for all $a \in \cA$ and $\varphi \in \cD(\pi_\rho)$. Let $x \in \gX_G$ and $\psi \in \cD(\pi_\rho)$. Then $x$ is of the form $x=as^{-1}$ with $a \in \cA$ and $s \in \cS$ and $\varphi:=\rho(s^{-1})\psi$ belongs to $ \cD(\pi_\rho)$ by Lemma \ref{drho}(ii). Applying (\ref{defpi}) twice we derive
\begin{align*}
T\rho(x)\psi &=T\rho(as^{-1})\rho(s^{-1})^{-1}\varphi=T\pi_\rho(a)\varphi=\pi_\rho(a)T\varphi=\rho(as^{-1})\rho(s^{-1})^{-1}T\varphi\\ &=\rho(as^{-1})\pi_\rho(s)T\varphi =\rho(as^{-1})T\pi_\rho(s)\varphi =\rho(x)T\rho(s^{-1})^{-1}\varphi =\rho(x)T\psi.
\end{align*}
Since $\cD(\pi_\rho)$ is dense in $\cH$, $T\rho(x)=\rho(x)T$. Because $\gX_G$ generates the algebra $\cX$, $T$ is in the commutant $\rho(\gX)^\prime$. Conversely, if $T$ is in $\rho(\cX)^\prime$, it follows at once from the definitions (\ref{domrho}) of $\cD(\pi_\rho)$ and (\ref{defpi}) of $\pi_\rho$ that $T$ belongs to $\pi_\rho(\cA)^\prime_s$.
\end{proof}
\begin{remk}
The Mittag-Leffler lemma  used in the proof of Lemma \ref{drho} even states that $E_\infty=\cD(\pi_\rho)$ is dense in each Hilbert space $(E_n,||\cdot ||_n)$ for $n \in \dN$. This implies that $\cD(\pi_\rho)$ is core for each operator $\rho(s^{-1})^{-1}$ for $s\in \cS$.
\end{remk}
\begin{remk}\label{wellbeh} A fundamental
problem in unbounded representation theory of $*$-algebras is to select and to classify classes of "well-behaved" $*$-representations among the large variety of representations. An approach to this problem have been proposed in \cite{SS}. Fraction algebras give another possibility by defining
well-behaved $*$-representations of the $*$-algebra $\cA$ as those of the form $\pi_\rho$. Propositions \ref{schrod2} and \ref{cond3} below support such a definition.
\end{remk}
\begin{exam}\label{exam2}
Retain the notation of Example \ref{exam1} and assume that $\rho$ is torsionfree, that is, $E((0,0))=0$. Recall that $x=ba^{-1}$ in the $*$-algebra  $\cA\cS_O^{-1}$. For $\varphi \in \cD(\pi_\rho)=\cap_{n=1}^\infty \rho(a^n)\cH$ we have $\pi_\rho(x)\varphi=\rho(b)\rho(a)^{-1}\varphi$
and
\begin{align*}
\pi_\rho(q(x))\varphi=\int_\cC q(\mu \lambda^{-1} )~ dE(\lambda,\mu)\varphi,~~ q\in \dC[x].
\end{align*}
\end{exam}

\section{Abstract Strict Positivstellens\"atze}

In addition to condition $(O)$  we now essentially use the following assumption:\\
$(AB)$ {\textit{The $*$-algebra $\gX$ is algebraically bounded, that is, for each $x\in \cX$ there is a positive number $\lambda$ such that $\lambda {\cdot} 1-x^*x \in \sum \cX^2.$}}

Note that condition $(AB)$ implies that all $*$-representations of $\cX$ act by bounded operators.

The three theorems in this section are abstract strict Positivstellens\"atze for the $*$-algebra  $\cA$.
\begin{thm}\label{abstpos1}
Suppose that conditions (O) and (AB) are satisfied.
Let $a\in \cA_h$. Suppose there is an element $t \in \cS$ such that $t^{-1}a(t^*)^{-1} \in \gX$ and the following assumptions are fulfilled:\\
(i) For each irreducible $\cS^{-1}$-torsionsfree $*$-representation $\rho$ of $\gX$ on a Hilbert space $\cH(\rho)=\cD(\rho)$ there exists a bounded self-adjoint operator $T_\rho >0$ on $\cH(\rho)$ such that $ \pi_\rho(a) \geq T_\rho$ .\\
(ii) $ \rho_{tor}(t^{-1}a(t^*)^{-1}) >0$ for each irreducible $\cS^{-1}$-torsion $*$-representation $\rho_{tor}$ of $\gX$.\\
Then there exists an element $s\in \cS_O$ such that $s^*as \in \sum \cA^2$.
\end{thm}
The following simple lemma is used in the proofs of Theorems \ref{abstpos1} and \ref{mar}.
\begin{lemma}\label{ybtlemma}
Suppose that $b \in \cA$, $r\in \cS$ and $x:=r^{-1}b(r^*)^{-1} \in \cX$. Then for any $\cS^{-1}$-torsionfree $*$-representation $\rho$ of $\cX$ we have
\begin{align}\label{ybt}
\langle \rho(x) \varphi,\varphi \rangle=  \langle \pi_\rho(b)\rho((r^*)^{-1}) \varphi, \rho((r^*)^{-1})\varphi \rangle,~\varphi \in \cD(\pi_\rho).
\end{align}
\end{lemma}
\begin{proof}
By our assumption $(O)$ there are elements $t\in \cS$ and $y \in \cX$ such that $rx=yt=b(r^*)^{-1}$. If $\varphi \in \cD(\pi_{\rho})$, then $\psi:= \rho(t^{-1})^{-1}\varphi \in \cD(\pi_\rho)$
by Lemma \ref{drho}(iii).
Using these facts  we compute
\begin{align*}
&\langle \rho(x) \varphi,\varphi \rangle= \langle \rho(r^{-1}yt)\rho(t^{-1})\psi,\varphi \rangle = \langle \rho(r^{-1}y)\psi,\varphi \rangle =
\langle \rho(y)\psi,\rho((r^*)^{-1})\varphi \rangle =\\ &\langle \rho(b(tr^*)^{-1}) \psi, \rho((r^*)^{-1})\varphi \rangle
= \langle \pi_\rho(b)\rho((tr^*)^{-1}) \psi, \rho((r^*)^{-1})\varphi \rangle \\&=\langle \pi_\rho(b)\rho((r^*)^{-1}) \rho(t^{-1})\psi, \rho((r^*)^{-1})\varphi \rangle =\langle \pi_\rho(b)\rho((r^*)^{-1}) \varphi, \rho((r^*)^{-1})\varphi \rangle
\end{align*}
where the fifth equality follows from formula (\ref{defpi}), because we have $y=b (tr^*)^{-1} \in \gX$.
\end{proof}
\noindent
Proof of Theorem \ref{abstpos1}:

 Set $y:=t^{-1}a(t^*)^{-1}$. Our first aim is to show that $y \in \sum \gX^2$. The proof of this assertion is based on a now standard separation argument which is has been first used in  \cite{S2}, see e.g. Sections 5.1 and 5.2 in \cite{S5} for the noncommutative case.

 Assume to the contrary that $y$ is not in $\sum \gX^2$.
Since $\gX$ is algebraically bounded by assumption (AB), the unit element $1$ of $\gX$ is an algebraic inner point of the wedge $\sum~\gX^2$ of the real vector space $\gX_h$. Therefore, by Eidelheit's separation theorem (see e.g. \cite{J}, 0.2.4), there exists an $\dR$-linear functional $F\not\equiv 0$ on $\gX_h$ such that $
F(y)\leq 0$ and $F(\sum \gX^2) \geq 0$. There is no loss of generality to assume that $F(1)=1$. By a standard application of the Krein-Milman theorem (see e.g. \cite{J}, 0.3.6 and 1.8.3) it follows that this functional $F$ can be choosen to be extremal (that is, if $G$ is another $\dR$-linear functional on $\cX_h$ such that $G(y)\leq 0$, $G(1)=1$ and $F(x)\geq G(x)\geq 0$ for all $x\in \sum \cX^2$, then $G=F$). We extend $F$ to a $\dC$-linear functional, denoted also by $F$, on $\gX$. Then $F$ is an extremal state of the $*$-algebra $\gX$. Let $\rho_F$ be the $*$-representation of $\gX$ which is associated with $F$ by the GNS-construction. Using once more that $\cX$ is algebraically bounded, it follows that all operators $\rho_F(x)$, $x \in \gX$, are bounded, so we can assume that $\cD(\rho_F)=\cH(\rho_F)$. Since the state $F$ of $\cX$ is extremal, $\rho_F$ is irreducible. Therefore, by the decomposition of $\rho_F$ discussed in Section 2, $\rho_F$ is either an $\cS^{-1}$-torsion or an $\cS^{-1}$-torsionfree $*$-representation.

The crucial step of this proof is to show that $\rho_F(y) >0$. If $\rho_F$ is torsion, then we have $\rho_F(y) >0$  by assumption (ii). Now we suppose that $\rho:=\rho_F$ is torsionfree. Combining equation (\ref{ybt}) in Lemma \ref{ybtlemma}, applied with $a=b$, $r=t$, $x=y$,  and the assumption $\pi_{\rho}(a)\geq T_{\rho}$, we obtain
\begin{align}\label{yat}
&\langle \rho(y) \varphi,\varphi \rangle
\geq \langle T_{\rho}  \rho((t^*)^{-1})\varphi, \rho((t^*)^{-1})\varphi \rangle
\end{align}
for $\varphi \in \cD(\pi_\rho)$. Because $\rho(y)$, $T_\rho$ and $\rho((t^*)^{-1})$ are bounded operators and $\cD(\pi_\rho)$ is dense in $\cH(\rho)$ by Lemma \ref{drho}, it follows that equation (\ref{yat}) holds for arbitrary vectors $\varphi \in \cH(\rho)$. Since $T_\rho >0$ and $\ker~\rho((t^*)^{-1})=\{0\}$  because $\rho=\rho_F$ is torsionfree,  (\ref{yat}) implies that $\rho_F(y)>0$.

Thus we have  $\rho_F(y)>0$ as just shown and $F(y)\leq 0$ by construction. Since $F\not\equiv 0$, this is the desired contraction. Therefore,  $y \in \sum \gX^2$.

 We write $y$ as a finite sum $\sum_i y_i^*y_i$ with $y_i \in \gX$. From the Ore property of the set $\cS_O$ it follows that for all elements $y_it^* \in \cA\cS_O^{-1}$ there is a common right denominator, that is, there exist elements $s \in \cS_O$ and $a_i \in \cA$ such that $y_it^*=a_is^{-1}$ for all $i$. Then $y =t^{-1}a(t^*)^{-1} =\sum_i y_i^*y_i$ implies that $a = \sum_i (s^*)^{-1}a_i^* a_i s^{-1}$ and so  $s^*as = \sum_i a_i^*a_i \in \sum \cA^2$. $\hfill \Box$

\medskip
 Assuming the stronger condition $(IA)$ instead of $(O)$ we have the following stronger result.
\begin{thm}\label{abstpos2}
Asumme that conditions $(IA)$ and $(AB)$ are satisfied.
Let $a\in \cA_h$. Suppose there is an element $t \in \cS$ such that $t^{-1}a(t^*)^{-1} \in \gX$ and the following assumptions are fulfilled:\\
(i) For each irreducible $\cS^{-1}$-torsionfree $*$-representation $\rho$ of $\gX$ on a Hilbert space $\cH(\rho)=\cD(\rho)$  there exists a bounded self-adjoint operator $T_\rho >0$ on $\cH(\rho)$ such that $ \pi_\rho(a) \geq T_\rho$ .\\
(ii) $ \rho_s(t^{-1}a(t^*)^{-1}) >0$ for each irreducible $*$-representation $\rho_s$ of the $*$-algebra $\gX_s$ and $s\in \cS_G$.\\
Then there exists an element $s\in \cS_O$ such that $s^*as \in \sum \cA^2$.
\end{thm}
\begin{proof}
Since condition $(IA)$ holds by assumption, $(O)$ is satisfied by Lemma \ref{A12} and each torsion $*$-representation $\rho_{tor}$ is a direct sum of representations $\rho_s$ of $\cX$ such that $\rho_s(s^{-1})=0$ for $s \in \cS_G$ by Lemma  \ref{decomp}. Therefore, assumption (ii) above implies assumption (ii) of Theorem \ref{abstpos1}, so the assertion follows from Theorem \ref{abstpos1}.
\end{proof}
\begin{remk} Let $a$ be an element of $\cA$ satisfiying assumption (i) of Theorems \ref{abstpos1} or \ref{abstpos2}. By (\ref{axs}) there exists $t \in \cS$ such that $t^{-1}a(t^*)^{-1} \in \gX$.  Moreover, if $t^{-1}a(t^*)^{-1} \in \gX$ and $s \in \cS_G$, then $(st)^{-1}a((st)^*)^{-1} \in \cJ_s$ and so $\rho_s((st)^{-1}a((st)^*)^{-1})=0$ for any $*$-representation $\rho_s$ of $\cX_s$. Hence it is crucial in  both theorems  to find an element $t$ for which assumption (ii) holds as well.
\end{remk}
\begin{remk}\label{trivcase}
 Let us consider the trivial case when $\cS=\{1\}$. Then we have $\cA=\cX$ and $(O)$ trivially holds. Hence  Theorem \ref{abstpos1} gives the following asssertion (see e.g. \cite{S5}, Proposition 15) for an \textit{algebraically bounded} $*$-algebra $\cX$:  \textit{Let $a\in\cX_h$. If for each irreducible $*$-representation $\rho $ of $\cX$ there is a positive  number $\varepsilon $ such that $\rho(a) \geq \varepsilon $, then $a \in \sum \cX^2$.}
\end{remk}
The next theorem works only with representations $\pi_\rho$ of $\cA$. It can be considered as a non-commutative version of M. Marshall's extension of the Archimedean Positivstellensatz to noncompact semi-algebraic sets \cite{M2}.
\begin{thm}\label{mar}
Assume that $(O)$ and $(AB)$ hold. Let $a \in \cA_h$ and $t\in \cS$ be such that $y:=t^{-1}a(t^*)^{-1} \in \gX$. Then we have: \\
(i) If $\pi_\rho(a) \geq 0$ for all irreducible $\cS^{-1}$-torsionfree $*$-representations $\rho$ of $\cX$, then for each $\varepsilon >0$ there exists $s_\varepsilon \in \cS_O$ such that
$s_\varepsilon^*(a+ \varepsilon tt^*)s_\varepsilon \in \sum \cA^2$.\\
(ii) If for any $\varepsilon>0$ there is an element $s_\varepsilon \in \cS$ such that
$s_\varepsilon(a+ \varepsilon tt^*)s_\varepsilon^* \in \sum \cA^2$, then $\pi_\rho(a)\geq 0$ for all $\cS^{-1}$-torsionfree $*$-representations $\rho$ of $\cX$.
\end{thm}
\begin{proof}
(i): Suppose that $\varepsilon >0$. Since $\pi_\rho(a)\geq 0$ , we conclude from equation (\ref{ybt}), applied with $b=a$, $r=t$, that $\rho(y)\geq 0$ on $\cD(\pi_\rho)$ and by the density of $\cD(\pi_\rho)$ on $\cH(\rho)$. Thus $y+ \varepsilon $ satisfies the assumption of the Positivstellensatz in Remark \ref{trivcase}.
 Therefore, $y + \varepsilon \in \cX^2$, that is, $y+\varepsilon =\sum_i y_i^*y_i$ where $y_i \in \cX$. Proceeding as in the last paragraph of the proof of Theorem \ref{abstpos1}, there exist elements $s_\varepsilon \in \cS_0$ and $a_i\in \cA$ such that $y_it^*=a_is_\varepsilon^{-1}$ for all $i$ and we get $s_\varepsilon t(y+\varepsilon)(s_\varepsilon t)^*= s_\varepsilon^*(a+ \varepsilon tt^*)s_\varepsilon =\sum_i a_i^*a_i \in \sum \cA^2$.

(ii): Assume that $s_\varepsilon \in \cS$ and $c:=s_\varepsilon t(y+\varepsilon)(s_\varepsilon t)^*=s_\varepsilon(a+ \varepsilon tt^*)s_\varepsilon^* \in \sum \cA^2$. Therefore, since $\pi_\rho$ is $*$-representation of $\cA$, $\pi_\rho(c)\geq 0$. Equation  (\ref{ybt}), applied with $b=c$, $r=s_\varepsilon t$, $x=y+\varepsilon$, yields that $\rho(y+\varepsilon)\geq 0$ on $\cD(\pi_\rho)$ and so on $\cH(\rho)$. Since $\varepsilon >0$ was arbitrary, we have $\rho(y) \geq 0$. Combining the latter with identity (\ref{ybt}), now applied with $b=a$, $r=s_\varepsilon$, $x=y$, it follows that $\pi_\rho(a)\geq 0$.
\end{proof}

\section{Multi-Graded $*$-Algebras}

In this section we assume that the $*$-algebra $\cA$ has a multi-degree map
$d:\cA\backslash\set{0}\to\dN_0^k$ satisfying the following
conditions for arbitrary non-zero $a,b\in\cA$ and $\lambda\in\dC:$

$(d1)$ $d(\lambda a)=d(a)$ and $d(a+b)\leq d(a)\vee d(b),$

$(d2)$ $d(ab)=d(a)+d(b),$

$(d3)$ $d(a^*)=d(a)$,\\
where $a+b\neq 0$ in $(d1)$ and  we use the following notations for multi-indices $\gn{=}(n_1,\dots,n_k),$\\
$\gm{=}\{ m_1 ,\dots,m_k)\in\dZ^k:$
$\gn\vee\gm=(\max(m_1,n_1),\dots,\max(m_k,n_k))$,\\
$\gn\leq\gm$ if $n_1\leq m_1,\dots,n_k\leq m_k$,
$\gn<\gm$ if $n_1<m_1,\dots, n_k<m_k.$

We extend the map $d$ to a multi-degree map $d$ of $\cA\cS_O^{-1}$
to $\dZ^k$ by putting $d(as^{-1})=d(a)-d(s)$ for
$a\in\cA\backslash\set{0}$ and $s\in\cS_O.$ It is straightforward
to check that $d$ is well-defined and that conditions $(d1)$--$(d3)$ hold for the algebra $\cA\cS_O^{-1}$ as well.

Further, we suppose that the following conditions are valid:

$(A3)$ {\textit{$d([a,s])\leq d(a)$ for all $s \in \cS_G$ and $a \in \cA_G$.}}

$(A4)$ {\textit{$as^{-1} \in \cX$ for all $s \in \cS_G$ and $a \in \cA$ such that $d(a)\leq d(s)$.}}

$(A5)$ {\textit{For $a \in \cA$ and $\gn,\gk \in \dN_0^k$ such that $d(a)\leq \gn+\gk$ there exist finitely many elements
$b_i,c_i \in \cA$ satisfying $d(b_i)\leq \gn$, $d(c_i)\leq \gk $ for all $i$ and $a=\sum_{i}b_ic_i$.}

\begin{lemma}\label{a4}
$(i)$ $d([a,s])\leq d(a)$ for  $s \in \cS_G$ and $a \in \cA$.\\
$(ii)$ $s^{-1}at^{-1} \in \cX$ for $s,t \in \cS$ and $a \in \cA$ such that $d(a)\leq d(st)$.\\
$(iii)$ $a(ts)^{-1}b \in \cX$ for  $s,t \in \cS_G$ and $a,b \in \cA$ such that $d(ab)\leq d(ts)$, $d(s){=}d(t)$ and $d(a)\leq d(s)$.\\
$(iv)$ If $\cS=\cS_O$, then $as^{-1}b \in \cX$ for $s \in \cS$ and $a,b \in \cA$ such that $d(ab)\leq d(s)$.
\end{lemma}
\begin{proof}
(i): Let $\cB$ denote the set of all $a \in \cA$ for which the assertion of (i) is true. By conditions $(d1)$ and $(d3)$, $\cB$ is a $*$-invariant linear subspace of $\cA$. Suppose that $b_1,b_2 \in \cB$ and $s \in \cS_G$. Using conditions $(d1)$ and $(d2)$ and the fact that $d([b_l,s])\leq d(b_l)$, $l=1,2$,  we obtain
$$d([b_1b_2,s])=d(b_1[b_2,s]+[b_1,s]b_2)\leq (d(b_1)+d([b_2,s])\vee (d([b_1,s])+d(b_2))\leq d(b_1)+d(b_2)=d(b_1b_2),
$$
so $\cB$ is a $*$-algebra. Since it contains all generators  of $\cA$ by assumption $(A3)$, we have $\cB=\cA$.

(ii): We first treat the case $s=1$. Suppose that the assertion is valid for some $t \in \cS$ and all $a \in \cA$. By induction it suffices to show that it holds then for the element $ts_j$ of $\cS$, where $s_j\in \cS_G$. Let $a$ be an element of $\cA$ such that $d(a)\leq d(ts_j)$. By assumption $(A5)$ we can assume without loss of generality that $a=bc$, where $d(b)\leq d(s_j)$ and $d(c)\leq d(t)$. Note that $bs_j^{-1}\in \cX$ by $(A4)$.  Since $d(c)\leq d(t)$, we have $d([c,s_j])\leq d(t)\leq d(ts_j)$ by (i) and hence $[c,s_j](ts_j)^{-1} \in \cX$ and $ct^{-1}$ by the induction hypothesis. Therefore, it follows from the identity
$$bc(ts_j)^{-1}= bs_j^{-1}(ct^{-1}-[c,s_j](ts_j)^{-1})
$$
that $bc(ts_j)^{-1}\in \cX$. This completes the proof of (ii) in the case $s=1$.

Suppose now that $d(a) \leq d(st)$. Again by $(A5)$ we can asumme that $a=bc$, where $d(b)\leq d(s)$ and $d(c)\leq d(t)$. Then $d(b^*) \leq d(s^*)$ by $(d3)$. By the preceding paragraph we have $s^{-1}b =(b^*(s^*)^{-1})^* \in \cX$ and $ct^{-1}\in \cX$, so that $s^{-1}at^{-1}= s^{-1}bct^{-1}\in \cX$.

(iii): It suffices to check that all three summands on the right hand side of the identity
$$a(ts)^{-1}b = s^{-1}abt^{-1}+ (s^{-1}a)(t^{-1}[b,t]t^{-1}) +  (s^{-1}[a,s])(ts)^{-1}b$$
belong to $\cX$. Indeed, the first one is in $\cX$ by (ii). Since $d([a,s])\leq d(a)$ by (i) and $d(a)\leq d(s)$ by assumption, the elements $s^{-1}a$ and $s^{-1}[a,s]$ are in $ \cX$  by (ii). Since $d[b,t])\leq d(b) \leq d(t^2)=2d(t)=d(ts))$ by (i) and by the assumption $d(s)=d(t)$, we have $t^{-1}[b,t]t^{-1}\in \cX$ and $(ts)^{-1}b \in \cX$ once again by (ii). Hence the second and the third summands are also in $\cX$.

(iv): By the assumption $\cS= \cS_O$, there exist elements $t\in \cS$ and $c \in \cA$ such that $s^{-1}b=ct^{-1}$. Since then  $-d(s){+}d(b)=d(c){-}d(t)$, we have $d(ac)=d(a){+}
d(c)\leq d(t)$ and hence $as^{-1}b =act^{-1} \in \cX$ by (ii).
\end{proof}

\begin{lemma}\label{t2}
Let $\rho$ be a $*$-representation of a $*$-algebra $\cB$ and $b \in \cB$. If $\rho((b^*b)^m)=0$ for some $m \in \dN$, then $\rho(b)=0.$
\end{lemma}
\begin{proof}
Upon multiplying by some appropriate power $(b^*b)^k$ we can assume that $m=2^n$ for some $n \in \dN_0$.
If $m=1$, then $||\rho(b)\varphi||^2 =\langle \rho(b^*b)\varphi, \varphi \rangle =0$ for $\varphi \in \cD(\rho)$ and hence $\rho(b)=0$. By induction the same reasoning shows that the assertion holds for all numbers $m$ of the form $m=2^n$, where $n \in \dN_0$.
\end{proof}

\begin{lemma}\label{rho}
Let $c \in \cA$, $s \in\cS_G$ and $m \in \dN$. Suppose that $d(c) \leq (2m{-}1)(d(s){-}d(c))$.
Then we have $((cs^{-1})^*(cs^{-1}))^m \in \cX s^{-1}$ and $\rho_s(cs^{-1})=0$ for each $*$-representation $\rho_s$ of the quotient $*$-algebra $\cX_s=\cX/\cJ_s$.
\end{lemma}
\begin{proof} First note that $cs^{-1}\in \cX$ by Lemma \ref{a4}(ii), since $d(c)\leq d(s)$ by assumption.

We define a sequence of multi-indices $\gn_j{=}(n_{j1},\dots,,n_{jk})$, $j{=}1,{\dots},m$. If $m{=}1$, put $\gn=2d(c)$. Now let $m\geq 2$. Fix $l\in \{1,\dots,k\}$. If $d(s)_l\geq 2 d(c)_l$, we set $n_{j l}=2d(c)_l$ for $j{=}1,\dots,m$. Suppose that $d(s)_l \le 2s(c)_l$. Then there exists a number $m_l\in \{2,\dots,m\}$ such that
\begin{align}\label{ml}
(2m_l-3)(d(s)_l-d(c)_l) \leq d(c)_l \leq (2m_l-1)(d(s)_l-d(c)_l)
\end{align}
Define $n_{1l}=d(s)_l$, $n_{jl}= 2d(c)_l$ if $m_l\leq j\leq m$ and
\begin{align*}
 n_{jl}=2(j-1)(d(s)_l{-}d(c)_l) + d(s)_l~~{\rm if}~~ 2\leq j \leq m_l{-}1.
\end{align*}
Using the preceding definitions we  verify that
\begin{align}\label{c1}
d(c) \leq \gn_j\leq 2d(c)~{\rm for}~ j=1,\dots,m,\\\label{c2}
 2d(c){-}\gn_{j-1} +\gn_j \leq 2d(s)~{\rm for}~ j=2,\dots,m.
\end{align}
Indeed, for $j{=}1$, we have $d(c)_l\leq n_{1l}=d(s)_l\leq 2d(c)_l$. If $2 \leq j\leq m_l{-}1$, using the first inequality of (\ref{ml}) we derive
$$n_{jl}\leq 2(m_l{-}2)(d(s)_l-d(c)_l) + d(s)_l =(2m_l{-}3)(d(s)_l-d(c)_l) + d(c)_l\leq 2d(c)_l$$
and from the definition of $n_{jl}$ we obtain
$$ n_{jl} \geq n_{2l} = 2(d(s)_l- d(c)_l) + d(s)_l \geq d(c)_l.
$$
This proves (\ref{c1}). If $2\leq j \leq m_l-1$, we have $2d(c)_l-n_{j-1,l}+n_{jl} = 2d(s)_l$ by the above definitions. If $j=m_l$, then the corresponding definitions and the second inequality of (\ref{ml}) yield
\begin{align*}
 2d(c)_l-n_{j-1,l}+n_{jl} &= 2d(c)_l -2(m_l{-}2)(d(s)_l-d(c)_l) -d(s)_l + 2d(c)_l \\&=  d(c)_l -(2m_l{-}1)(d(s)_l-d(c)_l) +2d(s)_l \leq 2d(s)_l
\end{align*}
which proves (\ref{c2}).

Now we write the element $((cs^{-1})^*(cs^{-1}))^m$ of $\cX$ in the form
\begin{align}\label{am}
((cs^{-1})^*(cs^{-1}))^m =t^{-1} A_1 (ts)^{-1} A_2 (ts)^{-1} \cdots A_{m} s^{-1},
\end{align}
where $t:=s^*$ and $A_1=\dots=A_{m}:=c^*c$. Note that $d(s)=d(t)$ and $d(A_j)=2d(c) \leq d(ts)$.

 If $m{=}1$, then $d(A_1)=2d(c) \leq d(s){=}d(t)$ and hence $t^{-1}A_1 \in \cX$ by Lemma \ref{a4}(ii).

 Now suppose that $m\geq 2$. For $j{=}1,\dots,m{-}1$, set $\gk_j:=2d(c){-}\gn_{j}$. By the second inequality of (\ref{c1}), we have $\gk_j \in \dN_0^k$. By definition, $\gn_j+\gk_j =2d(c)=d(A_j)$. Therefore, by condition $(A5)$ we can write the element $A_j$ of $\cA$ as a finite sum $A_j= \sum_i b_{ji}c_{ji}$ of elements $b_{ji}, c_{ji} \in \cA$ such that $d(b_{ji})\leq \gn_j$ and $d(c_{ji})\leq \gk_j$. Since $\gn_1{=}d(t)$ by definition, $t^{-1}b_{1i}\in \cX$ by Lemma \ref{a4}(ii). If $j{=}2,\dots,m{-}1$, then we have $\gk_{j-1} + \gn_{j} =2d(c)- \gn_{j-1} +\gn_j \leq 2d(s)=d(ts)$ by (\ref{c2}) and $\gk_j=2d(c)-\gn_j \leq d(c)\leq d(s)$ by the first inequality of (\ref{c1}). Therefore, Lemma \ref{a4}(iii) applies and yields that $c_{j-1,i} (ts)^{-1}b_{j,i^\prime}\in \cX$. Finally, we have $(ts)^{-1}A_m \in \cX$, since $\gn_m=2d(c){=}d(A_m)$ by construction.

  In the preceding two paragraphs we have shown that $t^{-1} A_1 (ts)^{-1} A_2 (ts)^{-1} \cdots A_{m}\in \cX$. Therefore, by (\ref{am}) the element $((cs^{-1})^*(cs^{-1}))^m$ belongs to $ \cX s^{-1}\subseteq \cJ_s$, so that $\rho_s(((cs^{-1})^*(cs^{-1}))^m)=0$. The second assertion follows from Lemma \ref{t2} applied to $b=cs^{-1}$.
\end{proof}
\begin{remk}\label{extval} The preceding proof shows that the assertion of Lemma \ref{rho} is valid for $s \in \cS$ (rather than $s \in \cS_G$) provided that $a(s^*s)^{-1}b \in \cX$ for all $a,c \in \cA$ satisfying $d(a)\leq d(s)$ and $d(ab) \leq 2d(s)$.
\end{remk}
The next three propositions contains results about elements which are annihilated by the representations $\rho_s$ of the quotient $*$-algebras $\cX_s$.
\begin{prop}\label{rhonull1}
Let $s,t \in \cS_G$ and $a \in \cA$ be such that $d(a) {<}d(st)$. Then  $\rho_s(s^{-1}at^{-1})=\rho_s(t^{-1}as^{-1})=0$ for any $*$-representation $\rho_s$ of the $*$-algebra  $\cX_s=\cX/\cJ_s$.
\end{prop}
\begin{proof}
The assumption $d(a) < d(st)$ implies that $d(a)_l< d(s)_l+d(t)_l$ for $l=1,\dots,k$. We choose $\gn, \gk \in \dN_0^k$ such that $n_l+k_l=a_l$, $n_l \leq d(t)_l$ and $k_l< d(s)_l$ for $l{=}1,\dots,k$. Since $d(a)= \gn+\gk$, by condition $(A5)$ we can write  $a=\sum_i b_ic_i$, where $b_i,c_i \in \cA$, $d(b_i)\leq \gn$ and $d(c_i)\leq \gk$. Since $d(c_i)_l\leq k_l < d(s)_l$, there is a number $m\in \dN$ such $m(d(s){-}d(c_i))\geq d(c_i)$ for all $i$. Then we have $t^{-1}b_i, c_is^{-1} \in \cX$ by Lemma \ref{a4}(ii) and $\rho_s(c_i s^{-1})=0$ by Lemma \ref{rho}, so that $\rho_s(t^{-1}as^{-1})=\sum_i \rho_s(t^{-1}b_i)\rho_s(c_is^{-1})=0$.
\end{proof}
\begin{prop}\label{rhonull2}
Suppose that $\cS=\cS_O$. Let $s{=}s_1\dots s_p\in \cS$ and $t{=}s_{p+1}\dots s_{p+q}  \in \cS$, where $s_l \in \cS_G$ for $l{=}1,\dots,p+q$. If  $a \in \cA$ and $d(a) <d(st)$, then we have  $\rho_{s_l}(s^{-1}at^{-1})=\rho_{s_l}(t^{-1}as^{-1})=0$ for each $*$-representation $\rho_{s_l}$ of the $*$-algebra  $\cX_{s_l}=\cX/\cJ_{s_l}$, $l{=}1,\dots,p{+}q$.
\end{prop}
\begin{proof}
Let us carry out the proof of $\rho_{s_l}(t^{-1}as^{-1})=0$ for $l{=}1,\dots,p$. The other assertions are derived in a similar manner. We argue as in the preceding proof of Proposition \ref{rhonull1} and retain the notation used therein.  Since $\cS=\cS_O$, it follows from Lemma \ref{a4}(iv) and Remark \ref{extval} that the assertion of Lemma \ref{rho} is valid for $s$ and $c_i$, that is, we have $((c_is^{-1})^*(c_is^{-1}))^m\in \cX s^{-1}\subseteq \cJ_{s_l}$. Hence $\rho_{s_l}(c_is^{-1})=0$ by Lemma \ref{t2} which in turn implies that $\rho_{s_l}(t^{-1}as^{-1})=\sum_i \rho_{s_l}(t^{-1}b_i)\rho_{s_l}(c_is^{-1})=0$.
\end{proof}
 For the next proposition we need one more notation. Let $s \in \cS$, $r\in \cS_G$ and $a \in \cA$. We say that $r$ is a factor of $s$ if there are elements $s_1,\dots,s_p\in \cS_G$ and $i\in \{1,\dots,p\}$ such that $s=s_1\dots s_p$ and $r=s_i$. We shall write $a <_r s$ if $r$ is a factor of $s$ and there are multi-indices $\gr, \gn \in \dN_0^k$ such that $d(a)=\gr+\gn$, $\gr< d(r)$ and $\gn \leq d(s){-}d(r)$.

\begin{prop}\label{rhonull3}
Suppose that $\cS=\cS_O$. Let $s,t \in \cS$, $r\in \cS_G$
 and $a \in \cA$. Assume that $r$ is a factor of $s$ or a factor of $t$. If $a <_r st$, then $\rho_r(s^{-1}at^{-1})=\rho_r(t^{-1}as^{-1})=0$ for each $*$-representation $\rho_r$ of $\cX_r=\cX/ \cJ_r$.
\end{prop}
\begin{proof}
The proof follows by some  modifications in the proofs of Lemma \ref{rho} and Proposition \ref{rhonull1}. We explain this for the proof of $\rho_r(t^{-1}as^{-1})=0$ and in the case where $r$ is a factor of $s$, say $s=s_1\dots s_p$ and $r=s_i$.

First we modify the proof of Lemma \ref{rho}. Let $c$ be an element of $\cA$ such that $c <_r s$. We write $d(c)=\gr +\gn$ with $\gr< d(r)$ and $\gn \leq d(s){-}d(r)$. Since $\gr < d(r)$, there exists an $m \in \dN$ such that $\gr \leq (2m-1)(d(r){-}\gr)$. We construct a sequence of multi-indices $\gn_j$ as in the proof of Lemma \ref{rho} with $d(c)$ replaced by $\gr$ and $d(s)$ replaced by $d(r)$ therein. Then equations (\ref{c1}) and (\ref{c2}) yield $\gn_j \leq 2\gr$ and $2\gr-\gn_{j-1}+\gn_j \leq 2d(r)$. Put
$\gk_j:=2\gr-\gn_j$. We now decompose $A_j=c^*c$ as a finite sum $A_j=\sum_i b_{ji}c_{ji}$ with $d(b_{ji}) \leq \gn_j + d(c){-}d(r) $ and $d(c_{ji})\leq \gk_j +d(c){-}d(r)$. Then we obtain
$$d(c_{j-1,i}b_{ji^\prime})\leq \gk_{j-1}+\gn_j +2d(c)-2d(r)= 2\gr-\gn_{j-1}+\gn_j+2d(c)-2d(r)\leq 2d(c)\leq d(st).
$$
Since we assumed that $\cS=\cS_O$, Lemma \ref{a4}(iv) applies and yields that  $c_{j-1,i}(ts)^{-1}b_{ji^\prime}\in \cX$. In a similar manner we obtain that $t^{-1}b_{1i}\in \cX$. Recall that $\gn_m=2\gr$ and $\gk_m=0$ by construction. Therefore we have $d(c_{mi}) \leq d(c){-}d(r)$ and so $d(rc_{mi})\leq d(c)\leq d(s)$. Employing again Lemma \ref{a4}(iv) we get $rc_{mi}s^{-1}\in \cX$ and so $c_{mi}s^{-1}= r^{-1}(rc_{mi}s^{-1}) \in \cJ_r$. Combining  the latter with (\ref{am}) it follows that $((cs^{-1})^*(cs^{-1}))^m\in \cJ_r$. Hence we obtain $\rho_r(cs^{-1})=0$ by Lemma \ref{t2}.

Since $a<_r d(st)$, as in the proof of Proposition \ref{rhonull1} we decompose $d(a)=\gn + \gk$, where $\gn \leq d(t)$, $\gr<\gk \leq d(s)$, and $\gr < d(r)$. By (A6) we can write and $a=\sum b_lc_l$ with $d(b_l)\leq \gn$ and $d(c_l)\leq \gk$. Since $r$ is a factor of $s$, we have $c_i <_r s$ and hence $\rho_r(c_ls^{-1})=0$ as shown in the preceding paragraph. Because of $t^{-1}b_l \in \cX$ by Lemma \ref{a4}(ii), we conclude that $\rho_r(t^{-1}as^{-1})=0$. \end{proof}

\section{Application: A Strict Positivstellensatz for the Weyl Algebra}

Throughout this section і$\cA$ denotes the Weyl algebra $\cW(1)$, that is, $\cA$ is the unital $*$-algebra with hermitian generators $p$ and $q$ and defining relation
\begin{align}\label{pq}
pq-qp =- i 1.
\end{align}
It is well-known that this commutation relation is satisfied by the
self-adjoint operators
$$
(P_0\varphi)(t) =-i \varphi^\prime (t)~~{\rm and}~~ (Q_0 \psi)(t) = t \psi(t),~t \in \dR,
$$
on the Hilbert space $L^2(\dR)$. The pair $(P_0,Q_0)$ is called \textit{Schr\"odinger pair} and the
corresponding $*$-representation $\pi_0$ of the $*$-algebra $\cA$ is the \textit{Schr\"odinger representation}. That is,
$$
(\pi_0(p)\varphi)(t) =-i \varphi^\prime (t)~{\rm and}~ (\pi_0 (q)\varphi)(t) = t \varphi (t)~~{\rm for}~~ \varphi \in \cD(\pi_0){=}\cS(\dR)\subseteq \cH(\pi_0){=}L^2(\dR).
$$
We fix two non-zero reals $\alpha$ and $\beta$ and put
$$\cS_g=\{s_1=p-\alpha i ,~ s_2=q-\beta i \},~\cS_G=\cS_g\cup \cS_g^*,~ \gX_G
=\cS_G^{-1},~ \cA_G=\{ p,~q \}.
$$
From the
relation (\ref{pq}) it follows immediately that the $*$-monoid $\cS$ generated by $\cS_G$ is an Ore set, that is, we can assume that $\cS{=}\cS_O$. The unital $*$-subalgebra $\cX$ of $\cA\cS_O^{-1}$ is generated by $x:=s_1^{-1}$ and $y:=s_2^{-1}$.
From  (\ref{pq}) we easily derive
 the following relations in the $*$-algebra $\cX$:
\begin{align}\label{xybounded}
x-x^* =2i\alpha ~ x^*x,~~ y-y^*= 2i \beta ~ y^*y,\\ \label{normal}
xx^*= x^*x,~ yy^* =y^*y,\\
 \label{xycomm}
xy-yx = -i xy^2x =-i yx^2 y, ~xy^*-y^*x = -i x(y^*)^2x =-i y^*x^2 y^*.
\end{align}
\begin{lemma}\label{cond1}
With the preceding definitions, conditions $(O)$, $(IA)$ and $(AB)$ are fulfilled.
\end{lemma}
\begin{proof} Let us prove $(AB)$.
 Using relations (\ref{xybounded}) it follows that
\begin{align}\label{abounded}
1- \alpha^2 x^*x = (1+i\alpha  x)^* (1+i\alpha  x)~{\rm and}~ 1- \beta^2 y^*y = (1+\beta i y)^* (1+\beta  i y)
\end{align}
are in $\sum \cX^2$, so conclude that $\cX=\cX_b$. This means that $\gX$ is algebraically bounded, so $(AB)$ is satisfied.

Condition $(IA)$ is easily derived from relations (\ref{xybounded})--(\ref{xycomm}) and condition $(O)$ follows from $(IA)$ according to Lemma \ref{A12}.
\end{proof}
\begin{lemma}\label{selfadjoint}
Let $\gamma \in \dR$ and let $z$ be a bounded normal operator on a Hilbert $\cH$ such that
$z-z^*= 2 \gamma i z^*z$ and ${\rm ker}~z=\{0\}$. Then $A:= z^{-1} + i\gamma  I$ is a self-adjoint operator on $\cH$.
 \end{lemma}
{\it{Proof.}}
First we note that ${\rm ker}~z^*=\{0\}$, because $z$ is normal.
Since $z^*=z(I-2\gamma i z^*)$ and $z=z^*(I+2 \gamma i z)$, we have $\cD((z^*)^{-1}) {=}z^* \cH {=}z\cH {=}\cD(z^{-1})$. Further, from the identity $z^*=z(I-2\gamma i z^*)$ we get $z^{-1}z^* =I- 2 i\gamma  z^*$ on $\cH$. For $\varphi =z^*\psi \in \cD((z^*)^{-1})$ we obtain  $z^{-1}\varphi =z^{-1}z^*\psi = \psi -2i \gamma z^*\psi = (z^*)^{-1} \varphi -2i\gamma  \varphi$, that is,
$z^{-1} \supseteq (z^*)^{-1}- 2\gamma iI$. Because $\cD((z^*)^{-1})=\cD(z^{-1})$ as noticed above, it follows that $z^{-1} =(z^*)^{-1}- 2i\gamma I$. Using the latter identity we derive
\begin{align*}
A= z^{-1} +i \gamma I= (z^*)^{-1} -i \gamma I= (z^{-1})^* - i \gamma I = (z^{-1}+i\alpha I)^*=A^*. \quad \quad\quad \quad\quad\quad\quad\quad \quad\quad \Box
\end{align*}

The assertion of the next proposition describes Schr\"odinger pairs in terms of resolvents. A slightly different characterization of this kind has been first obtained in \cite{B}.
\begin{prop}\label{schrod1}
Suppose that $x$ and $y$ are closed linear operators on a Hilbert space $\cH$ with trivial kernels (that is, ${\rm ker } ~x= {\rm ker}~y=\{0\}$) satisfying equations (\ref{xybounded})--(\ref{xycomm}). Then
\begin{align}\label{PQ}
P= x^{-1} + i\alpha  I~{\rm and}~Q = y^{-1} +\beta iI
\end{align}
 are self-adjoint operators on $\cH$ and the pair $(P,Q)$ is unitarily equivalent to a direct sum of Schr\"odinger pairs $(P_0,Q_0)$ on $L^2(\dR)$.
\end{prop}
\begin{proof} The self-adjointness of operators $P$ and $Q$ follows from Lemma  \ref{selfadjoint}.

From the first equations of (\ref{xycomm}) we conclude that $xy \cH=yx \cH$. Let us denote this vector space by $\cD$. Since $\cD(P){=}x\cH$ and $\cD(Q){=}y\cH$ by (\ref{PQ}), we have $\cD\subseteq \cD(PQ)\cap \cD(QP)$.

We show that $PQ\varphi -QP\varphi=-i\varphi$ for $\varphi \in \cD$. Indeed,
if $\varphi =yx \psi$, then by the first equations of (\ref{xycomm}) we derive
\begin{align*}
PQ\varphi- QP\varphi&= (P-i \alpha)(Q-i\beta )\varphi - (Q-i\beta )(P-i\alpha)\varphi \\&=(P-i \alpha)(Q-i\beta )yx\psi -(Q-i\beta )(P-i\alpha)xy(I+i yx)\psi \\ &=
\psi -(I+iyx)\psi=i\varphi.
\end{align*}
Moreover, from the definitions (\ref{PQ}) we obtain $(P-i \alpha)(Q-i\beta )\cD=(P-i \alpha)(Q-i\beta )yx\cH =\cH$ and
$(Q-i\beta )(P-i\alpha )\cD=(Q-i\beta )(P-i\alpha )xy\cH =\cH$.

By the preceding we have shown that $P$ and $Q$
 satisfy the assumptions of a theorem by T. Kato \cite{K2}. The assertion of this theorem states that
\begin{align}\label{weyl}
e^{i\lambda P} e^{i\mu Q} = e^{i \lambda \mu} e^{i \mu Q} e^{i \lambda P}
\end{align}
for nonnegative reals $\lambda$ and $\mu$. That (\ref{weyl}) holds for nonnegative reals obviously implies that (\ref{weyl}) is fulfilled for arbitrary reals $\lambda$ and $\mu$. Thus, $P$ and $Q$ are self-adjoint operators satisfying the Weyl relation. By the Stone--von Neumann uniqueness theorem (see e.g. \cite{Pu}, Theorem 4.3.1), the pair $(P,Q)$ is is unitarily equivalent to a direct sum of Schr\"odinger pairs $(P_0,Q_0)$.
\end{proof}

\begin{prop}\label{schrod2}
Suppose $\rho$ is an $\cS^{-1}$-torsionfree $*$-representation of the $*$-algebra $\cX$. Then the $*$-representation $\pi_\rho$ of $\cA$ is unitarily equivalent to a direct sum of Schr\"odinger representations.
\end{prop}
\begin{proof}
Since the $*$-algebra $\cX$ is algebraically bounded by Lemma \ref{cond1}, all operators of $\rho(\cX)$ are bounded. The operators $\rho(x)$ and $\rho(y)$ satisfy the relations (\ref{xybounded})--(\ref{xycomm}) and have trivial kernels because $\rho$ is torsionfree. Therefore, by Proposition \ref{schrod1} the pair $(P,Q)$ defined by (\ref{PQ}) (with $x$ and $y$ replaced by $\rho(x)$ and $\rho(y)$, respectively) is unitarily equivalent to a direct sum of Schr\"odinger pairs.
The map $\rho \to \pi_\rho$ according to Theorem \ref{assrep}
respects unitary equivalences and direct sums, so it suffices to prove the assertion in the case when $P=P_0$ and $Q=Q_0$ on the Hilbert space $L^2(\dR)$.
By (\ref{domrho}) the domain $\cD_\rho=\cD(\pi_\rho)$
is the intersection of ranges of all finite products of operators $(P-\alpha i)^{-1}=\rho(x) =\rho(s_1^{-1})$, $(Q- \beta i)^{-1}=\rho(y)=\rho(s_2^{-1})$ and their adjoints. Hence $\cD_\rho=\cD(\pi_\rho)$ is the Schwartz space $\cS(\dR)$ and for $\varphi \in \cD(\pi_\rho)$ we have
$$\pi_\rho(p- \alpha i)\varphi = \pi_\rho(s_1)\varphi=\rho(s_1^{-1})^{-1}\varphi = \rho(x)^{-1}\varphi = (P- \alpha i)\varphi=-i\varphi^\prime - \alpha i \varphi, $$
so  $\pi_\rho(p) \varphi =-i\varphi^\prime$. Similarly, $\pi_\rho(q)\varphi =t \varphi$. That is, $\pi_\rho$ is the Schr\"odinger representation.
\end{proof}

Now let $c$ be an arbitrary nonzero element of the Weyl algebra $\cA$. Because $\{p^nq^k; k,n \in \dN_0\}$ and $\{q^np^k; k,n \in \dN_0\}$ are vector space bases of $\cA$, we can write $c$ as
\begin{align}\label{crep1}
c = \sum_{j=0}^{d_1} \sum_{l=0}^{d_2}   \gamma_{jl}p^jq^l= \sum_{n=0}^{d_2} f_n(p)q^n =\sum_{k=0}^{d_1} g_k(q)p^k,
\end{align}
where $\gamma_{jl}$ are complex numbers and $f_n(p)\in \dC[p]$, $g_k(q)\in \dC[q]$ are polynomials all of them uniquely determined by $c$.  We choose $d_1$ and $d_2$ such that there are numbers $j_0, l_0\in \dN_0$ for which $\gamma_{d_1,l_0}\neq 0$ and $\gamma_{j_0,d_2}\neq 0$. Set $d(c)=(d_1,d_2)$. It is easily checked that $d$ defines a multi-degree on $\cA$ satifying conditions $(d1)$--$(d3)$ and $(A3)$--$(A5)$. Note that $f_{d_2}\neq0$ and $g_{d_1}\neq 0$.
\begin{thm}\label{pqpos}
Let $c=c^*$ be a nonzero element of the Weyl algebra $\cA$ with multi-degree $d(c)=(2m_1,2m_2)$, where $m_1,m_2\in \dN_0$. Suppose that:\\
(I) There exists a bounded self-adjoint operator $T > 0$ on $L^2(\dR)$ such that $\pi_0(c) \geq T$.\\
(II) $\gamma_{2m_1,2m_2}\neq 0$ and both polynomials $f_{2m_2}$ and $g_{2m_1}$ are positive on the real line.\\
Then there exists an element $s \in \cS$ such that $s^*cs \in \sum~\cA^2$.
\end{thm}
\begin{proof} Recall that $\cS=\cS_O$ and all results from Sections 4 and 5 apply, because the corresponding assumptions are fulfilled.
Set $t:=s_2^{m_2}s_1^{m_1}$. Since $d(c)=(2m_1,2m_2)=d(t^2)$, it follows from Lemma \ref{a4}(ii) that  $z:=t^{-1}c (t^*)^{-1}=x^{m_1}y^{m_2}c (y^*)^{m_2}(x^*)^{m_1}$ is in $\cX$.

The assertion will follow from Theorem \ref{abstpos2} once assumptions (i) and (ii) therein are established. Assumption (i) is a consequence of assumption (I), since the only irreducible $*$-representations $\pi_\rho$ is the Schr\"odinger representation $\pi_0$ by Proposition \ref{schrod2}.

We prove that $\rho_{s_1}(z)>0$ for each $*$-representation $\rho_{s_1}$ of $\gX_{s_1}$. (By Theorem \ref{abstpos2} we  could assume that $\rho_s$ is irreducible,
but this does not simplify our reasoning.) If $k< 2m_1$, then $d(g_k(q)p^k)_1 <2m_1$, so that $g_k(q)p^k <_{s_1} t^*t$ and hence $\rho_{s_1}(t^{-1}g_k(q)p^k(t^*)^{-1})=0$ by Proposition \ref{rhonull3}. Likewise,  $d(p^{m_1}g_{2m_1}(q)p^{m_1}{-} g_{2m_1}(q)p^{2m_1})_1 <2m_1$ and so $\rho_{s_1}(t^{-1}(p^{m_1}g_{2m_1}(q)p^{m_1}{-} g_{2m_1}(q)p^{2m_1})(t^*)^{-1})=0 $ again by Proposition \ref{rhonull3}.
Therefore, by (\ref{crep1}) we have
\begin{align}\label{rhos1}
\rho_{s_1}(z)=\rho_{s_1}(x^{m_1}y^{m_2}p^{m_1}g_{2m_1}(q)p^{m_1} (y^*)^{m_2}(x^*)^{m_1}).
\end{align}
Since $xy=yx(1{-}ixy)$ by (\ref{xycomm}),
$x^{m_1}y^{m_2}-y^{m_2}x^{m_1}$ and $(y^*)^{m_2}(x^*)^{m_1})-(x^*)^{m_1}(y^*)^{m_2}$ are linear combinations of terms $r^{-1}$, where $r\in \cS$ and $d(r)>(m_1,m_2)$. Hence from  Proposition \ref{rhonull3} we get
\begin{align}\label{rhos2}
 \rho_{s_1}(x^{m_1}y^{m_2}p^{m_1}g_{2m_1}(q)p^{m_1} (y^*)^{m_2}(x^*)^{m_1}-y^{m_2}x^{m_1}p^{m_1}g_{2m_1}(q)p^{m_1} (x^*)^{m_1}(y^*)^{m_2})=0.
\end{align}
From the relation $xp=1{+}\alpha ix$ it follows that $x^{m_1}p^{m_1}-1$ and $p^{m_1}(x^*)^{m_1}-1$ are linear combinations of $x^j=s_1^{-j}$, where $1\leq  j \leq m_1$. Therefore, we have
\begin{align}\label{rhos3}
\rho_{s_1}(z)=\rho_{s_1}(y^{m_2}x^{m_1}p^{m_1}g_{2m_1}(q)p^{m_1} (x^*)^{m_1}(y^*)^{m_2}-y^{m_2}g_{2m_1}(q)(y^*)^{m_2})=0
\end{align}
by Proposition \ref{rhonull3}. (All facts derived above from Proposition \ref{rhonull3} can be also verified directly by using the commutation rules between $p, q, x,$ and $y$.) Combining equations (\ref{rhos1})--(\ref{rhos3}) we obtain
$\rho_{s_1}(z) =\rho_{s_1}((yy^*)^{m_2}g_{2m_1}(q))$.
Let us write the polynomial $g_{2m_2}$ as $g_{2m_1}(q)=\sum_{l=0}^{2m_2} \gamma_l q^l$. Clearly, $\gamma_{2m_2}=\gamma_{2m_1,2m_2}$.  Since $q=y^{-1}+\beta i$ by the definition  of $y=s_2^{-1}$ and $y^*=y(1-2\beta iy^*)$ by (\ref{xybounded}),
\begin{align}\label{ymg}
\rho_{s_1}(z)=\rho_{s_1}((yy^*)^{m_2}g_{2m_1}(q))= (I-2\beta i \rho_{s_1}(y)^*)^{m_2} \sum_{l=0}^{2m_2} \gamma_l (I+\beta i \rho_{s_1}(y))^l \rho_{s_1}(y)^{2m_2-l}
\end{align}
is a polynomial, say $h(\rho_{s_1}(y))$, of the normal operator $\rho_{s_1}(y)$ and its adjoint. Hence the spectrum of $\rho_{s_1}(z)$ is the set of numbers $h({\sf y})$, where $\sf y$ is in the spectrum of  $\rho_{s_1}(y)$. Since $y-y^*=2\beta iy^*y$ by (\ref{xybounded}), $\sf y$ belongs to the circle ${\sf y}-{\bar {\sf y}}=\beta i {\bar{\sf y}}{\sf y}$ of the complex plane. If ${\sf y}=0$, then $h(0)=\gamma_{2m_2}=\gamma_{2m_1,2m_2}>0$ by assumption (II). If $\sf y$ is a nonzero number of this circle, then $\sf y$ is of the form $({\sf q}-\beta i)^{-1}$ with  ${\sf q} \in \dR$. Inserting this into (\ref{ymg}), we compute $h(\sf y)= (y{\overline y})^{m_2} g_{2m_1}(\sf q)$. Since $g_{2m_1}(\sf q) >0$ by assumption (II), we get $h(\sf y)> 0$. Thus we have shown that the spectrum of the normal operator $\rho_{s_1}(z)$ is contained in $(0,+\infty )$, so that $\rho_{s_1}(z)>0$.

A similar reasoning using the positivity of $f_{2m_2}$ instead of that of $g_{2m_1}$ yields $\rho_{s_2}(z)>0$. Hence assumption (ii) of Theorem \ref{abstpos2} is satisfied.
\end{proof}

\section{A Resolvent Approach to Integrable Representations of the Enveloping Algebra of the
$ax+b$-Group}

Throughout this and the next section we denote by $G$ the affine group of the line, that is, $G=\{(e^\gamma,\delta); \gamma, \delta \in \dR\}$ with multiplication rule $(e^{\gamma_1},\delta_1)(e^\gamma_2,\delta_2)=(e^{\gamma_1+ \gamma_2}, e^{\gamma_1} \delta_2 +\delta_1)$ and by $\mathfrak{g}$ the Lie algebra of the Lie group $G$. Recall that $\cg$  has a vector space basis $\{\sf a,\sf b\}$ satisfying the commutation relation $[\sf a,\sf b]=\sf b$. The exponential map $\rm exp$ of $\mathfrak{g}$ into $G$ is given by
${\rm exp}~ \gamma {\sf a}= (e^\gamma,0)$ and ${\rm exp}~ \gamma {\sf b}= (1,\gamma)$, where $\gamma \in \dR$.

We  need a few notions on Lie group representations (see e.g. \cite{S1}, Chapter 10, or \cite{W}, Chapter 4, for more details).
By a unitary representation of $G$ we mean a strongly continuous homomorphism $U$ of $G$ into the unitary group of a Hilbert space $\cH(U)$ and by $dU$ we denote the associated $*$-representation of the enveloping algebra $\cE(\cg)$ of the Lie algebra  $\cg$ on the dense vector space $\cD^\infty(U)$ of $C^\infty$-vectors of $U$. If ${\sf c} \in \cg$, then $\partial U(\sf c)$ denotes the infinitesimal generator of the unitary group $U(e^{\gamma {\sf c}})$, that is, $e^{\gamma \partial U(\sf c)}=U(e^{\gamma \sf c})$ , $\gamma \in \dR$, and we have $dU({\sf c})\varphi =\partial U(\sf c)\varphi$ for $\varphi \in \cD^\infty(U)$. Note that the operator $i \partial U(\sf c)$ is self-adjoint.

The next proposition
and its subsequent theorem
characterize integrable representations of the Lie algebra $\cg$ in terms of resolvents of the two generators.
\begin{prop}\label{dUAB}
Suppose that $U$ is a unitary reresentation of $G$. Let $\alpha $ and $ \beta$ be real numbers such that $|\alpha | >1$, $\beta \neq 0$, and set $x_0=(A-\alpha i)^{-1}$, $x_1=(A-(\alpha{+}1)i)^{-1}$ and $y=(B-\beta i)^{-1}$, where $A:=i\partial U({\sf a})$ and $B:=i\partial U(\sf b)$. Then we have the relations
\begin{align}\label{x0x1y}
x_0-x_0^* =2\alpha i~ x_0^*x_0&= 2 \alpha i~x_0x_0^*,~ y-y^*= 2 \beta i~ y^*y= 2 \beta i~ yy^*,\\
\label{x0resolvent}
&x_0-x_1=-i x_1x_0=-i x_0x_1,\\
\label{xy}
&x_0y-yx_1=-\beta yx_1x_0y .
\end{align}
\end{prop}
\begin{proof} Equations (\ref{x0x1y}) and (\ref{x0resolvent}) follow easily from the definitions of $x_0$, $x_1$ and $y$.

We  prove the commutation relation (\ref{xy}). From the relation $e^{-\gamma {\sf a}} e^{-\delta e^\gamma {\sf b}} =e^{-\delta {\sf b}} e^{-\gamma {\sf a}}$ in the group $G$ it follows that
\begin{align}\label{eAB}
e^{i\gamma A} e^{i\delta e^\gamma B} =e^{i\delta B} e^{i\gamma A}~~ {\rm for}~~ \gamma, \delta \in \dR.
\end{align}
First assume that $\beta <0$. Then, if $C$ is a self-adjoint operator, we have (see e.g. \cite{K1}, p. 482)
\begin{align*}
(C-\beta i)^{-1} =-i \int_0^\infty e^{\beta \lambda} e^{i\lambda C} d\lambda.
\end{align*}
Multiplying
(\ref{eAB}) by $e^{\beta \lambda}$ and integrating on $[0,+\infty)$ by using the preceding formula we get
\begin{align}\label{expresAB}
e^{i\gamma A}(e^{\gamma}B-\beta i )^{-1}=(B -\beta i)^{-1} e^{i\gamma A} ~{\rm for}~ \gamma \in \dR.
\end{align}
Applying the involution to (\ref{expresAB}) and multiplying then by $e^{-\gamma}$ it follows that formula (\ref{expresAB}) holds in the case $\beta >0$ as well. We now apply both sides  of (\ref{expresAB}) to a vector $\varphi\in \cD(A)$ and differentiate at $\gamma=0$. Then we obtain
\begin{align}
iA( B-\beta i )^{-1}\varphi -  B(B- \beta i)^{-2}\varphi = (B-\beta i)^{-1} iA \varphi.
\end{align}
Since $y=(B-\beta i)^{-1}$, the latter yields $(A-\alpha i)y\varphi -\beta y^2 \varphi=y(A-(\alpha {+}1) i)\varphi.$
If $\psi \in \cH$, then $\varphi:=x_1\psi \in \cD(A)$ and so $(A-\alpha i)yx_1 \psi -\beta y^2x_1 \psi = y\psi$. Multiplying by $x_0$ from the left we derive
\begin{align}\label{x0y2x1}
x_0y -yx_1= -\beta x_0y^2x_1,
\end{align}
so that $x_0y=(I-\beta x_0y)yx_1$. From the definitions of $x_0$ and $y$ it follows immediately that $||\beta x_0y|| \leq |\alpha|^{-1}< 1$. Therefore, we have $(I-\beta x_0y)^{-1}=\sum_{n=0}^\infty \beta^n (x_0y)^n$ and hence
\begin{align*}
yx_1=(I-\beta x_0y)^{-1}x_0y= \sum_{n=0}^\infty \beta^n (x_0y)^{n+1}.
\end{align*}
The latter implies that $(x_0y)yx_1=yx_1(x_0y)$. Inserting this into (\ref{x0y2x1}) we obtain (\ref{xy}).
\end{proof}
\begin{thm}\label{axbresol}
Let $\alpha, \beta \in \dR$, $\alpha< -1$ and $\beta \neq 0$. Suppose that $x_0$, $x_1$ and $y$ are bounded linear operators on a Hilbert space $\cH$  satifying the equations (\ref{x0x1y})--(\ref{xy}).
Assume that
${\rm ker}~ x_0={\rm ker}~y =\{0\}$ and define
\begin{align}\label{AB}
A:= x_0^{-1} + \alpha i I~~{\rm and }~~ B:= y^{-1} +\beta iI.
\end{align}
Then $A$ and $B$ are self-adjoint operators on $\cH$
and there exists a unitary representation $U$ of the group $G$ on $\cH$ such that $i \partial U({\sf a}) =A$ and $i \partial U({\sf b}) =B$.
\end{thm}
\begin{proof}
The basic pattern of the proof is similar to that of  Kato's theorem \cite{K2}, but the technical details are more complicated. The self-adjointness of $A$ and $B$ follows from Lemma \ref{selfadjoint}.

First we prove by induction on $n \in \dN$ that
\begin{align}\label{x01}
x_0^ny= yx_1^n+\beta i ~ y(x_1^n-x_0^n)y.
\end{align}
If $n{=}1$, then (\ref{x01}) holds by combining (\ref{xy}) and the first equality of (\ref{x0resolvent}). Suppose that (\ref{x01}) is valid for $n \in \dN$.  Note that $x_0x_1=x_1x_0$ by (\ref{x0resolvent}). Using first the induction hypothesis, then equation (\ref{x01}) in the case $n{=}1$ and finally once more the induction hypothesis, we compute
\begin{align*}
x_0^{n+1}y&= x_0(yx_1^n+\beta i ~ y(x_1^n-x_0^n)y)=
(yx_1+\beta i ~ y(x_1-x_0)y)(x_1^n+\beta i ~ (x_1^n-x_0^n)y)\\ &=yx_1^{n+1}+ \beta i~ y(x_1{-}x_0)yx_1^n + \beta i~ yx_1(x_1^n-x_0^n)y -\beta^2y(x_1{-}x_0)(x_1^n-x_0^n)y\\
&=
yx_1^{n+1}+ \beta i ~ y(x_1{-}x_0)(x_0^ny-\beta i ~ y(x_1^n-x_0^n)y) + \beta i ~ yx_1(x_1^n-x_0^n)y -\beta^2y(x_1{-}x_0)(x_1^n-x_0^n)y\\
&=yx_1^{n+1} +\beta i ~ y((x_1{-}x_0)x_0^n  +x_1(x_1^n-x_0^n))y
=yx_1^{n+1} +\beta i ~ y(x_1^{n+1}-x_0^{n+1})y,
\end{align*}
which completes the induction proof of equation (\ref{x01}).

Let $\cF_1$ denote the set of all complex $\lambda$ for which $\lambda$ and $\lambda +i$ are not real and
 the identity
\begin{align}\label{resident}
(A+i- \lambda )^{-n}y= y (A-\lambda )^{-n} + \beta i~ y((A-\lambda )^{-n} -(A+i -\lambda )^{-n})y
\end{align}
holds for all $n \in \dN$.
Suppose that $\lambda_0 \in \cF_1$.
Fix $k\in \dN$. Let $\lambda$ be a complex number such that $|\lambda-\lambda_0|(||(A-\lambda_0)^{-1}||+||(A+i-\lambda_0)^{-1}||)<1$. We multiply equation (\ref{resident}) by
$ {{n-1} \choose {k-1}}  (\lambda-\lambda_0)^{n-k}$ and sum over $n=k,k{+}1,\dots$.
Using the identities
\begin{align*}
(A-\lambda )^{-k} = \sum_{n=k}^\infty ~
 {{n-1} \choose {k-1}}
(\lambda -\lambda_0)^{n-k} (A-\lambda_0)^{-n},
\end{align*}
\begin{align*}
(A+i-\lambda )^{-k} = \sum_{n=k}^\infty ~
 {{n-1} \choose {k-1}}
(\lambda -\lambda_0)^{n-k} (A+i-\lambda_0)^{-n}
\end{align*}
we conclude that (\ref{resident}) is satisfied for $\lambda$ and $k$. Therefore, $\lambda \in \cF_1$ which proves that $\cF_1$
is open.
Recall that $(A-\alpha i)^{-1}=x_0$ by (\ref{AB}). Combining this fact with (\ref{x0resolvent}) we derive
$$
(A-\alpha i- i)x_1= (A-\alpha i)x_0(I+ix_1) - i x_1= I
$$
and similarly $x_1(A-\alpha i -i)=I$, so that $(A-\alpha i- i)^{-1}=x_1$. Inserting these formulas for $x_0$ and $x_1$ into (\ref{x01}) we obtain equation (\ref{resident}) for $\lambda=i+\alpha i$. That is, $i+\alpha i \in \cF_1$. Because $\cF_1$ is open as just shown,
the connected component of $i+\alpha i$ in the complement of $\dR \cup (\dR +i)$
is contained in $\cF_1$. Since $\alpha <-1$ by assumption, (\ref{resident}) holds for all $\lambda$ of the lower half-plane.

Multiplying (\ref{resident}) by $(-\lambda )^n$ and setting $\lambda = - n \gamma^{-1}i$ with $\gamma > 0$ and $n\in \dN$, we obtain
\begin{align}\label{an}
(I{-}\gamma n^{-1} i(A+i))^{-n}y  = y(I{-}\gamma n^{-1} iA)^{-n} +\beta i~ y((I{-}\gamma n^{-1} i A)^{-n} {-}(I{-}\gamma n^{-1} i(A+i))^{-n})y
\end{align}

We now need the following fact (see e.g. \cite{HPh}, p. 362 or \cite{K1}, p. 479):
If $C$ is the infinitesimal generator of a contraction semigroup $\{e^{\gamma C};\gamma\geq 0\}$, then we have
\begin{align}\label{gen}
e^{\gamma  C}= s{-}{\rm lim_{n \to \infty}} (I-\gamma n^{-1}C)^{-n}.
\end{align}
Applying this formula to the generators $iA$ and $i(A+i)$ of contraction semigroups, it follows from (\ref{an}) that
\begin{align*}
e^{\gamma i(A+i)}y =y e^{\gamma i A} +\beta i~ y(e^{\gamma i A}- e^{\gamma i(A+i)})y
\end{align*}
for all $\gamma >0$. Because $(B-\beta i)^{-1}=y$ by (\ref{AB}), the latter yields
\begin{align*}
e^{i\gamma A} y= ye ^{i\gamma  A}(e^{\gamma}(I+\beta i ~y)-\beta i~y)
=y e^{i\gamma  A}(e^{\gamma}B-\beta i)y.
\end{align*}
Hence we have
\begin{align*}
e^{i\gamma A}(e^{\gamma}B-\beta i)^{-1}= y e^{i\gamma A}=(B -\beta i)^{-1} e^{i\gamma A}
\end{align*}
which in turn implies that
\begin{align}\label{betan}
e^{i\gamma A}(e^{\gamma}B-\beta i)^{-n}=(B -\beta i)^{-n} e^{i\gamma A}
\end{align}
for all $n \in \dN$ and $\gamma >0$.

Now we fix $\gamma >0$ and consider the set $\cF_2$ of all $\lambda \in \dC\backslash \dR$ for which
\begin{align}\label{mu}
e^{i\gamma A}(e^{\gamma}B-\mu )^{-n}=(B -\mu )^{-n} e^{i\gamma A}
\end{align}
is satisfied for all $n \in \dN$. Arguing as in the paragraph before last, with $A$ and $A+i$ replaced by $B$ and $e^\gamma B$, we conclude that $\cF_2$ is open. Since $\beta i \in \cF_2$ by (\ref{betan}), $\cF_2$ contains the lower half-plane when $\beta <0$ resp. the upper half-plane when $\beta >0$. Let us first assume that $\beta <0$. Then (\ref{mu}) is valid for all $\mu$ such that ${\rm Im}~\mu <0$.

Proceeding  as above, we multiply equation (\ref{mu}) by $(-\mu )^n $ and set $\mu =-n \delta^{-1}i$ with $\delta >0$ and $n \in \dN$. Letting $n \to \infty$ by using formula (\ref{gen}) we obtain
\begin{align}\label{expAB}
e^{i\gamma A} e^{i\delta e^\gamma B} =e^{i\delta B} e^{i\gamma A}.
\end{align}
Up to now equation (\ref{expAB}) has been proved only for $\gamma >0$ and  $\delta >0$.
We now show that (\ref{expAB}) holds for arbitary real numbers $\gamma$ and $\delta$. First we note that (\ref{expAB}) is trivially fulfilled if $\gamma=0$ or $\delta =0$. Applying the involution to (\ref{expAB}) and multiplying the corresponding equation by $e^{i\gamma  A}$ from the left and from the right we get
$e^{i\gamma A} e^{-i\delta e^\gamma B} =e^{-i\delta B} e^{i\gamma A}$.
This shows that (\ref{expAB}) is valid for all $\gamma \geq 0$ and  $\delta \in \dR$. Applying the involution to (\ref{expAB}), with $\delta$ replaced by real $\eta$, and multiplying then by $e^{i\eta e^{\gamma} B}$ from the left and by $e^{i \eta B}$ from the right we derive $e^{-i\gamma A} e^{i\eta B} =e^{i\eta e^{\gamma}B} e^{-i\gamma A}$. Setting $\delta=\eta e^{\gamma}$ the latter yields $e^{-i\gamma A} e^{i\delta e^{-\gamma} B} =e^{i\delta B} e^{-i\gamma A}$ which means that (\ref{expAB}) holds for $\gamma \leq0$ and $\delta \in \dR$. Thus, equation (\ref{expAB}) is satisfied for all reals $\gamma, \delta $.

 The case when $\beta >0$ is treated a in similar manner replacing $\delta >0$ by $\delta <0$ in the preceding.

For $(e^\gamma,\delta)\equiv {\rm exp}~\delta {\sf b}~ {\rm exp}~ \gamma{\sf a} \in G$ we define $U((e^\gamma,\delta))=e^{-i\delta B}e^{-i\gamma A}$. A straightforward computation based on equation (\ref{expAB}) shows that $U$ is a homomorphism of $G$ into the unitary group of $\cH$. Hence $U$ is a  unitary representation of $G$ on $\cH$.
Clearly, $i\partial U({\sf a})=A$ and $i\partial U({\sf b})=B$.
\end{proof}
As a byproduct of the preceding considerations the next theorem gives an integrability criterion for Hilbert space representations of the Lie algebra $\cg$. Here the density condition (\ref{densitycond}) is the crucial assumption for the integrability of the representation. Note that it is not sufficent that $A$ and $B$ are selfadjoint operators satisfying  relation (\ref{relAB}) on a common core.
\begin{thm}\label{intdens}
$(i)$ Suppose that $U$ is a unitary representation of $G$. Let $\alpha, \beta$ be fixed real numbers such that $|\alpha| >1$ and $\beta \neq 0$. Let
$A=i\partial U(\sf a)$, $B=i \partial U(\sf b)$ and $\cD=(A-\alpha i)^{-1}(B-\beta i)^{-1}\cH(U)$. Then $\cD$ is dense in $\cH(U)=(B-\beta i)(A-\alpha i)\cD$ and we have
\begin{align}\label{relAB}
AB\varphi- BA\varphi =iB\varphi~~{\rm for}~~\varphi \in \cD.
\end{align}
$(ii)$ Suppose that $A$ and $B$ are self-adjoint operators on a Hilbert space $\cH$. Let $\alpha, \beta \in \dR$, $\alpha <-1$ and $\beta \neq 0$.
Assume that there is a linear subspace $\cD\subseteq \cD(AB)\cap \cD(BA)$ of $\cH$ such that (\ref{relAB}) holds and that
\begin{align}\label{densitycond}
(B-\beta i)(A-\alpha i)\cD ~~ or ~~ (A-(\alpha{+}1) i)(B-\beta i)\cD~~  is~ dense~ in ~~ \cH.
\end{align}
Then
there exists a unitary representation $U$ of $G$ such that $A=i\partial U(\sf a)$ and $B=i \partial U(\sf b)$.
\end{thm}
\begin{proof} We retain the notations $x_0=(A-\alpha i)^{-1}$, $x_1=(A-(\alpha {+}1)i)^{-1}$ and $y=(B-\beta i)^{-1}$.

(i): Recall that
equation (\ref{xy}) is satisfied by Proposition \ref{dUAB} and (\ref{AB}) holds by definition. Obviously, ${\rm ker}~(x_0y)^* =\{0\}$, so $\cD=x_0y\cH$ is dense in $\cH(U)$.

From (\ref{xy}) and (\ref{AB}) it follows that $\cD\subseteq \cD(AB)\cap \cD(BA)$.
If $\psi \in \cH$, then $\varphi :=x_0y \psi =yx_1(I-\beta x_0y)\psi \in \cD$ by (\ref{xy}). To prove that equation (\ref{relAB}) is valid we compute
\begin{align*}
AB\varphi- BA\varphi-iB\varphi &= (A- (\alpha {+}1)i))(B-\beta i )\varphi - (B-\beta i)(A-\alpha i)\varphi +\beta \varphi  \\&=(A-(\alpha {+}1)i)(B-\beta i)yx_1(I-\beta x_0y)\psi -(B-\beta i )(A-\alpha i)x_0y\psi +\beta x_0y \psi  \\ &=
(I-\beta x_0y)\psi- \psi + \beta x_0y\psi=0.
\end{align*}
(ii): Assume that $\cD_1=(B-\beta i)(A-\alpha i)\cD$ is dense in $\cH$. The case when $(A-(\alpha{+}1) i)(B-\beta i)\cD$ is dense is treated in a similar manner.
Let $\varphi \in \cD_1$. Then $\varphi=(B-\beta i)(A-\alpha i)\psi$ for some $\psi \in \cD$. By (\ref{relAB}) we have $\varphi= (A-(\alpha{+}1)i)(B-\beta i)\psi + \beta \psi$, so that $x_0y\varphi=\psi$ and $yx_1\varphi= \psi +\beta yx_1\psi$ which in turn yields that $x_0y\varphi=yx_1\varphi -\beta yx_1x_0y\varphi=yx_1\varphi -\beta yx_0x_1y\varphi$. Since $\cD_1$ is dense in $\cH$, we have $x_0y=yx_1 -\beta yx_0x_1y$ on $\cH$, that is, (\ref{xy}) holds. Since equations (\ref{x0x1y}) and (\ref{x0resolvent}) follow at once from the definitions of $x_0$, $x_1$ and $y$, Theorem \ref{axbresol} applies and gives the assertion.
\end{proof}

\section{Application:
A Strict Positivstellensatz for the Enveloping Algebra of the
$ax+b$-Group}

In this section $\cA$ is the complex universal enveloping algebra $\cE(\cg)$ of the Lie algebra $\mathfrak{g}$ of the affine group of the real line. Setting $a:=i {\sf a}$ and $b:=i {\sf b}$,  $\cA$ becomes the unital $*$-algebra with two hermitian generators $a$ and $b$ and defining relation
\begin{align}\label{ab}
ab-ba =ib.
\end{align}

Let us fix two  reals $\alpha$ and $\beta$ such that $\alpha< -1$, $\beta \neq 0$ and $\alpha$ is not an integer and set
$$\cS_g=\{s=b-\beta i,~s_n=a-(\alpha {+}n) i;~n \in \dZ\},~\cS_G=\cS_g \cup \cS_g^*,~
\cX_G=\cS_G^{-1},~ \cA_G=\{ a,~b \}.
$$
Using (\ref{ab}) we obtain $s_{n+1}b=bs_n$, $s_{n-1}^*b=bs_n^*$ for $n\in \dZ$,  $s^2a=(s(a-i)+\beta)s$ and $(s^*)^2a= (s^*(a-i)-\beta)s^*$. From these formulas it follows  that the unital monoid $\cS$ generated by the set $\cS_G$ is a $*$-invariant left Ore set, so we can assume that $\cS=\cS_O$.

The $*$-subalgebra $\cX$ of $\cA\cS_O^{-1}$ is the unital algebra generated by the elements $y{:=}s^{-1}$ and $x_n{:=}s_n^{-1}$, where $n \in \dZ$, and their adjoints.
In the $*$-algebra $\cX$ we have the following relations:
\begin{align}\label{xo1}
x_n{-}x_n^* & =2(\alpha{+}n) i~ x_n^*x_n=2(\alpha{+}n) i~ x_nx_n^*,~ y{-}y^*= 2 \beta i~ y^*y= 2 \beta i yy^*,\\ \label{xo2}
x_n{-}x_k&=(n{-}k)i x_nx_k=(n{-}k)i x_kx_k,~ x_n{-}x_k^*=(2 \alpha {+} k{+}n)i~x_nx_k^*=(2 \alpha {+} k{+}n)i~x_k^*x_n,
\\  \label{xo3}
x_ny{-}yx_{n+1}&={-}\beta yx_{n+1}x_ny =- \beta x_n y^2x_{n+1},~x_ny^*{-}y^*x_{n+1}=\beta y^*x_{n+1}x_ny^* = \beta x_n (y^*)^2x_{n+1}.
\end{align}
\begin{lemma}\label{cond2}
Conditions $(O)$, $(IA)$ and $(AB)$  are satisfied.
\end{lemma}
\begin{proof}
The proof is similar to that of Lemma \ref{cond1}. As a sample, we  verify  $(IA)$. Combining relations (\ref{xo2}) and (\ref{xo3}) we obtain $x_nx_k=x_kx_n$, $x_nx_k^*=x_k^*x_n$,
\begin{align*}
x_ny&=yx_{n+1}(1-\beta x_n y)=(1-\beta x_n y)y(1+i x_{n+1})x_n,\\ x_ny^*&=y^*x_{n+1}(1+\beta x_n y^*)=(1+\beta x_n y^*)y^*(1+i x_{n+1})x_n,\\
x_n^*y&=yx_{n-1}^*(1-\beta x_n^*y)=(1-\beta x_n^*y)y(1+ ix_{n-1}^*)x_n^*,\\
x_n^*y^*&=y^*x_{n-1}^*(1+\beta x_n^*y^*) =(1+\beta x_n^*y^*)y^*(1+ix_{n-1}^*)x_n^*
\end{align*}
for $n,k \in \dZ$. From these equations and their adjoints we conclude that $(IA)$ is fulfilled.
\end{proof}
\begin{prop}\label{cond3}
For any $\cS^{-1}$-torsionfree $*$-representation $\rho$ of the $*$-algebra $\cX$ there exists a unique unitary representation $U$ of the group $G$ such that $\pi_\rho =dU$. The representation $\rho$ is irreducible if and only if $U$ is irreducible.
\end{prop}
\begin{proof}
Since the relations (\ref{x0x1y})--(\ref{xy}) are contained in (\ref{xo1})--(\ref{xo3}), Theorem \ref{axbresol} applies. Hence there exists a unitary representation $U$ of $G$ such that $i\partial U({\sf a})=A$ and $i\partial U({\sf b})=B$. As in the proof of Proposition \ref{schrod2} it follows that $\pi_\rho(a)\varphi= A\varphi$ and $\pi_\rho(b)\varphi=B\varphi$ for $\varphi \in \cD(\pi_\rho)$ and that $\cD(\pi_\rho)$ is the intersection of ranges of all finite products of operators $(A -i(\alpha{+}n))^{-1}=\rho(x_n) =\rho(s_n^{-1})$, $(B-\beta i )^{-1}=\rho(y)=\rho(s_2^{-1})$ and their adjoints. The latter set is obviously the intersection of domains of all finite products of $A$ and $B$. Hence  $\cD(\pi_\rho)$ is equal to the domain $\cD^\infty (U)$
 (see e.g. \cite{S1}, Theorem 10.1.9) of $dU$. Since $dU(a)\psi =i\partial U({\sf a})\psi = A\psi$ and $dU(b)\psi =i\partial U({\sf b})\psi=B\psi$ for $\psi \in \cD^\infty(U)$, we conclude that $\pi_\rho= dU$.

As stated in Theorem \ref{assrep}, $\rho$ is irreducible if and only if $\pi_\rho$ is so. But $dU=\pi_\rho$ is known to be irreducible if and only if the unitary representation $U$ is irreducible (\cite{S1}, 10.2.18).
\end{proof}

From Proposition \ref{dUAB} it follows easily the converse of Proposition \ref{cond3} is also true (that is, any $*$-representation $dU$ of $\cA$ is equal to $\pi_\rho$ for some torsionfree $*$-representation $\rho$ of $\cX$), but we will need this result in what follows.

Because $\{a^nb^k; k,n {\in} \dN_0\}$ and $\{b^na^k; k,n {\in} \dN_0\}$ are bases of the vector space $\cA$ by the Poincare-Birkhoff-Witt theorem, each nonzero element $c \in \cA$ can be written as
\begin{align}\label{crep2}
c = \sum_{j=0}^{d_1} \sum_{l=0}^{d_2}   \gamma_{jl}a^jb^l= \sum_{n=0}^{d_2} f_n(a)b^n =\sum_{k=0}^{d_1} g_k(b)a^n.
\end{align}
Here $\gamma_{jl}\in \dC$ and $f_n(a)$ and $g_k(b)$ are complex polynomials uniquely determined by $c$.  We define $d(c)=(d_1,d_2)$ if there are numbers $j_0, l_0\in \dN_0$ such that $\gamma_{d_1,l_0}\neq 0$ and $\gamma_{j_0,d_2}\neq 0$. Then $d$ is a multi-degree map on the $*$-algebra $\cA$. It is easily checked that conditions $(A3)$--$(A5)$ are valid.

\begin{thm}\label{abpos}
Supoose that $c=c^*$ is a nonzero element of the enveloping algebra $\cA=\cE(\cg)$ with multi-degree $d(c)=(2m_1,2m_2)$, where $m_1,m_2\in \dN_0$, satisfying the following assumptions:\\
(I) For each irreducible unitary representation $U$ of $G$ there exists a bounded self-adjoint operator $T_U > 0$ on $\cH(U)$ such that $dU(c) \geq T_U$.\\
(II) $\gamma_{2m_1,2m_2}\neq 0$ and the polynomials $f_{2m_2}(\cdot+m_2i)$ and $g_{2m_1}$ are positive on the real line.\\
Then there exists an element $s \in \cS$ such that $s^*cs \in \sum~\cA^2$.
\end{thm}
\begin{proof} Since the proof follows a similar pattern as the proof of Theorem \ref{pqpos}, we sketch only the necessary modifications.
Setting $t:=s^{m_2}s_0^{m_1}$, the element $z:=t^{-1}c (t^*)^{-1}=x_0^{m_1}y^{m_2}c (y^*)^{m_1}(x_0^*)^{m_1}$ belongs to $\cX$ by Lemma \ref{a4}(ii).

The assertion follows from
Theorem \ref{abstpos2}. It remains to prove that assumptions  (i) and (ii) therein are satisfied. Assumption (i) is a consequence of assumption (I) combined with Proposition \ref{cond3}. To verify assumption (ii) we first note that all ideals $\cJ_{s_n}$
coincide by relation (\ref{xo2}), so it suffices to show that assumption (II) implies that $\rho_s(z)>0$ and $\rho_{s_0}(z)>0$.

Let us begin with $\rho_s(z)$. Note that that we have $f_{2m_2}(a)b^{m_2}=b^{m_2}f_{2m_2}(a+m_2 i)$ by the commutation relation (\ref{ab}). Further, we have $yb=1+\beta i y$ and $by^*=1-\beta i y^*$. Using these facts and arguing as in the proof of Theorem \ref{pqpos} it follows that
\begin{align}\label{f2m}
\rho_{s}(z)=
\rho_{s}(x_0^{m_1}y^{m_2}f_{2m_2}(a)b^{2m_2}
(y^*)^{m_2}(x_0^*)^{m_1})=
\rho_{s}(x_0^{m_1}f_{2m_2}(a+m_2i)(x_0^*)^{m_1}).
\end{align}

Now we turn to $\rho_{s_0}(z)$. From (\ref{xo2}) and (\ref{xo3}) we have $x_0y-yx_0=(i-\beta y-\beta ix_1y)x_1$. As in the proof of Theorem \ref{pqpos} we therefore obtain
\begin{align*}
\rho_{s_0}(z)=\rho_{s_0}(x_0^{m_1}y^{m_2}g_{2m_1}(b)a^{2m_1} (y^*)^{m_2}(x_0^*)^{m_1})
=\rho_{s_0}(y^{m_2}x_0^{m_1}g_{2m_1}(b)a^{2m_1} (x_0^*)^{m_1}(y^*)^{m_2})
\end{align*}
Let $g_{2m_1}(b)=\sum_{l=0}^{2m_2} \gamma_l b^l$.
From (\ref{ab}) it follows that $g_{2m_2}(b)a^{m_1}= \sum_{l=0}^{2m_2} \gamma_l
(a-l i)^{m_1} b^l$. Moreover, $xa=1 +\alpha i x$. Using these relation we derive
\begin{align}\label{g2m}
\rho_{s_1}(z)=\rho_{s_0}(y^{m_2}x_0^{m_1}g_{2m_1}(b)a^{2m_1} (x_0^*)^{m_1}(y^*)^{m_2})=\rho_{s_1}(y^{m_2}g_{2m_1}(b)(y^*)^{m_2})
\end{align}

Having (\ref{f2m}) and (\ref{g2m}) a similar reasoning as in the last part of the proof of Theorem \ref{pqpos} shows that assumption (II) implies that $\rho_s(z)>0$ and $\rho_{s_0}(z)>0$.
\end{proof}

\begin{remk}\label{irraxb}
According to a classical result due to Gelfand and Naimark \cite{GN}, the set of equivalence classes of irreducible unitary representations of the group $G$ consists of two infinite-dimensional representations $U_{\pm}$ and of  a family $U_\gamma$, $\gamma \in \dR$, of one-dimensional representations. The associated infinitesimal representations $dU_\pm$ act on the domain
\begin{align*}
\cD^\infty(U_\pm)=\{f \in C^\infty(\dR):e^{nx}f^{(m)}(x) \in L^2(\dR)~ {\rm  for~ all}~ n,m\in \dN_0\}
\end{align*}
of the Hilbert space $L^2(\dR)$ by $dU_\pm(a)f =i f^\prime$ and $dU_\pm(b)f=\pm e^x f(x)$. For $\gamma \in \dR$ we have $dU_\gamma(a)=\gamma$ and $dU_\gamma(b)=0$. Inserting these expressions into (\ref{crep2}) leads to a more explicit form of assumption (I) of Theorem \ref{abpos}. That is, (I) is equivalent to the requirements $f_0 >0$ on $\dR$ and
\begin{align*}
dU_\pm(c) = \sum_{k=0}^{2m_1} g_k(\pm e^x )i^k  \big(\frac{d}{dx}\big)^k \geq T_\pm ~ {\rm on}~ \cD^\infty(U_\pm)
\end{align*}
for some bounded selfadjoint opertors $T_\pm$ on $L^2(\dR)$ satisfying $T_\pm >0$.
\end{remk}

\bibliographystyle{amsalpha}

\end{document}